\documentclass[11pt,a4paper]{article}
\usepackage{inputenc}
\usepackage{amsfonts}
\usepackage{amsthm}
\usepackage{amsmath}
\usepackage{amscd}
\usepackage{latexsym}
\usepackage{amssymb}
\usepackage[all]{xy}

\newtheorem{theorem}{Theorem}[section]

\newtheorem{example}[theorem]{Example}

\newtheorem{lemma}[theorem]{Lemma}

\newtheorem{proposition}[theorem]{Proposition}

\newenvironment{remarknf}{
  \refstepcounter{theorem}
  \noindent\textbf{Remark \thetheorem.}
}  

\def\d{\displaystyle}
\def\RR{{\mathbb{R}}}
\def\NN{{\mathbb{N}}}

\def\CC{{\mathbb{C}}}

\def\BK{{\mathcal{K}\left(\mathcal{H}\right)}}

\def\noi{\noindent}
\def\h{\mathcal{H}}
\def\k{\mathfrak{K}}
\def\B{\mathcal{B}(\h)}
\def\BI{\mathcal{B}(\mathfrak{I})}
\def\BK{\mathcal{B}(\mathfrak{K})}
\def\I{\mathfrak{I}}
\def\U{\mathcal{U}}
\def\UI{\mathcal{U}_{\mathfrak{I}}}

\def\UK{\mathcal{U}_{\mathfrak{K}}}
\def\o{\mathcal{O}}
\def\g{\mathcal{G}}
\newcommand{\PI}[2]{\left\langle #1 , #2 \right\rangle}

\def\noi{\noindent}
\def\h{\mathcal{H}}
\def\k{\mathfrak{K}}
\def\B{\mathcal{B}(\h)}
\def\BI{\mathcal{B}(\mathfrak{I})}
\def\I{\mathfrak{I}}
\def\U{\mathcal{U}}
\def\UI{\mathcal{U}_{\mathfrak{I}}}
\def\UIH{\mathcal{U}_{\mathfrak{I}}} 
\def\UIP{\mathcal{U}_{\mathfrak{I}}\left(P\right)}

\def\o{\mathcal{O}}
\def\g{\mathcal{G}}
\DeclareMathOperator{\y}{and}

\DeclareMathOperator{\rank}{rank}

\begin{document}


\title{\vspace*{0cm} Geometry of  unitary orbits of pinching operators }
\date{}
\author{ Eduardo Chiumiento and Mar\'ia E. Di Iorio y Lucero\footnote{All authors partially supported by Instituto Argentino de Matem\'atica and CONICET.}}

\maketitle

\abstract{\footnotesize{\noindent
 Let $\I$ be a symmetrically-normed ideal of the space of bounded operators acting on a Hilbert space $\h$. Let $\{ \, p_i \, \}_1  ^w$ $(1\leq w \leq \infty)$  be   a family of mutually orthogonal projections on $\h$.  The pinching operator associated with the former family of projections is given by 
\[  P: \I \longrightarrow \I, \, \, \, \, \, \, \, \,  P(x)=\sum_{i=1}^{w} p_i x p_i. \]
Let $\UI$ denote the Banach-Lie group of the unitary operators whose difference with the identity belongs to $\I$.  We study several geometric properties of the orbit
$$\UI(P)=\left\{\, L_{u} P L_{u^*} \, :\,  u \in  \UI \, \right\},$$ 
where $L_u$ is the left  representation of $\UI$ on the algebra $\BI$ of bounded operators acting on $\I$. The results include necessary and sufficient conditions for $\UIP$ to be a submanifold of $\BI$. Special features arise in the case of  the ideal $\k$ of compact operators. In general, $\UK(P)$   turns out to be a non complemented submanifold of $\BK$. We find a necessary and sufficient condition for $\UK(P)$ to  have complemented tangent spaces in $\BK$.
 We also show that $\UIP$ is a  covering space of another  natural orbit of $P$. A quotient Finsler metric is introduced, and the induced rectifiable is studied. In addition, we give an application of the results on $\UIP$  to the topology of the $\UI$-unitary orbit of a compact normal operator.}
}


\section{Introduction}

Let $\h$ be an infinite dimensional separable Hilbert space and $\B$ the space of bounded linear operators acting on $\h$.  We denote by $\U$  the group of unitary operators on $\h$.
Let $\Phi$  be a symmetric norming function  and  $\I=\mathfrak{S}_{\Phi}$ the corresponding symmetrically-normed ideal of $\B$ equipped with the norm $\| \, . \, \|_{\I}$. Let $\UI$ denote the group of unitaries which are perturbations of the identity by an operator in $\I$, i.e.
$$ \UI =\{  \, u \in  \U \, : \, u -1   \in  \I \, \}.  $$
It is a real Banach-Lie  group with the topology defined by the metric $d(u_1, u_2)= \| u_1 -u_2 \|_{\I}$, and its Lie algebra equals  
\[  \I_{sh}=\{   \,  x \in \I  \, : \, x^*=-x  \, \}, \]
which is the real Banach space of skew-hermitian operators in $\I$ (see \cite{B}).


Let $\{ \, p_i \, \}_1  ^w$ ($1\leq w \leq \infty$) be  a family of mutually orthogonal hermitian projections in $\B$. We do  not make any assumption on the sum of all the projections of the family,  so we could have that the projection $p_0:=1-\sum_{i=1}^w p_i$ is nonzero.     
The \textit{pinching operator} associated with $\{ \, p_i \, \}_1  ^w$ is defined by
\begin{equation*}\label{operators}
 P: \I \longrightarrow \I, \, \, \, \, \, \, \, \,  P(x)=\sum_{i=1}^{w} p_i x p_i, 
\end{equation*} 
where in case $w=\infty$  the  series is convergent in the uniform norm. Let $\BI$ denote  the Banach algebra of bounded operators acting on  $\I$. Left multiplication defines the bounded linear operators $L_x: \I \longrightarrow \I$, $L_x(y)=xy$, for $x \in \B$ and $y \in \I$. The left  representation of $\UI$ on  $\BI$, namely
$\UI \longrightarrow \BI$, $u \mapsto L_u$, allows us to introduce the following orbit
$$\UI(P):=\left\{\, L_{u} P L_{u^*} \, :\,  u \in  \UI \, \right\}.$$ 
The aim of this paper is to study  geometric properties of this  orbit.  Since every pinching operator is a continuous projection,  the present work might be regarded as a modest contribution to the vast literature on the differential and metric geometry of unitary orbits of projections in different settings (see e.g. \cite{AL8, AS, BR07b, cpr, cpr0, Up}).
Despite of some usual geometric properties that have already been studied in the afore-mentioned papers and still hold in this special orbit, we will also show some new special features of $\UIP$, especially concerning with its submanifold structure.


Pinching operators generalize the so-called notion of pinching of block  matrices developed in matrix analysis (see e.g. \cite{Bh}). In the framework of symmetrically-normed ideals, these operators have been studied in \cite{GK,S}.
If $\I$ is the trace class ideal,  pinching operators arise in quantum mechanics due to a well-known postulate of von Neumann on the measurement of density operators \cite{vN}. More recently, they have been shown to be examples of the quantum reduction maps introduced in \cite{RO}.

Let us describe the contents of the paper. 

In Section 2 we recall some basic facts on symmetrically-normed ideals, pinching operators and submanifolds of Banach manifolds.
Section 3 is devoted to the study of the differential structure of $\UIP$. For any symmetrically-normed ideal $\I$ different from the compact operators, we describe in Theorem \ref{complemento} several equivalent conditions to $\UI(P)$ be a submanifold of $\BI$. For the ideal $\k$ of compact operators many of these conditions are no longer equivalent. In fact, $\UK(P)$ is always a quasi submanifold of $\BK$, which rarely has complemented tangent spaces in $\BK$ (see Theorem \ref{compact quasi}).

In Section 4 we go further into the topological structure  of $\UIP$. We show that $\UIP$ is a covering space of another   natural orbit of $P$. The methods of this section use those of \cite{AS}, where a similar situation arises in relation  with the unitary orbit of a conditional expectation in von Neumann algebras.

The Section 5 is concerned with the metric structure of $\UIP$. Motivated by similar results on other homogeneous spaces \cite{AL9,C}  we study the rectifiable distance induced by quotient Finsler metric on $\UIP$. Under the assumption that the quotient topology on $\UI(P)$ coincide with the inherited topology from $\BI$, we prove that the rectifiable distance defines these topologies.  As a by-product we find that  $\UIP$ is complete with the rectifiable distance.

In Section 6 we study the topology of $\UI$-unitary orbits of  a compact normal operator. These type of unitary orbits  may be endowed with the quotient topology, though there is another quite natural topology, the one defined by the norm of the ideal $\I$. We show that both topologies coincide if and only if the compact operator  has finite rank. The proof makes use of the previous results on the topology of $\UIP$. 
This result is related with several works \cite{AL9, ALR,  BR, Bo}, where under different assumptions, the finite rank condition appears as sufficient to the statement on the topologies.

\section{Preliminaries}\label{intro}

\noi \textbf{Symmetrically-normed ideals.}
We begin  with some basics facts on symmetrically-normed ideals. For a deeper discussion of this subject we refer the reader to \cite{GK} or \cite{S}.

Let $\h$ be a Hilbert space.  No confusion will arise if $\| \, . \, \|$  denotes the norm of vectors in $\h$ and the uniform norm in $\B$. For $\xi, \eta \in \h$, let  $\xi \otimes \eta$ be the rank one operator defined by $(\xi \otimes \eta)(\zeta)=\PI{\zeta}{\eta}\xi$, for $\zeta \in \h$.   
  By a \textit{symmetrically-normed ideal} we mean a two-sided ideal $\I$ of $\B$ endowed with a norm $\| \, . \, \|_{\I}$ satisfying 
\begin{itemize}
\item $(\I, \| \, . \, \|_{\I})$ is a Banach space.
\item $\|xyz\|_{\I}\leq \|x\| \|y\|_{\I} \| z\|$, for  $x,z \in \B$ and $y \in \I$.
\item $\| \xi \otimes \eta \|_{\I}= \| \xi \| \, \| \eta \|$, for $\xi , \eta \in \h$.  
\end{itemize}
A result that goes back to J. Calkin (\cite{Cal}) states the inclusions $\mathfrak{F} \subseteq \I \subseteq \k$, where $\mathfrak{F}$ is the set of all the finite rank operators, $\I$ is a two-sided ideal of $\B$ and $\k$ the ideal of compact operators on $\h$. 
   
Symmetrically-normed ideals are closely related to the following class of norms. Let $\hat{c}$ be the real vector space consisting of all sequences with a finite number of nonzero terms.  A \textit{symmetric norming function} is a norm $\Phi: \hat{c} \to \RR$ satisfying the following properties: 
\begin{itemize}
\item $\Phi(1,0,0,\ldots)=1$.
\item $\Phi(a_1, a_2 , \ldots , a_n, 0, 0 , \ldots)=\Phi(|a_{j_1}|, |a_{j_2}| , \ldots , |a_{j_n}|, 0, 0 , \ldots)$, where  $j_1, \ldots, j_n$  is any permutation of the integers $1,2,\ldots, n$ and $n\geq 1$.  
\end{itemize}
Any symmetric norming function $\Phi$ gives rise to two symmetrically-normed ideals. Indeed, for any  compact operator $x$ one may consider the sequence $(s_n(x))_{n}$ of its singular values arranged in  non-increasing order, and thus define 
\[  \| x \|_{\Phi}:= \sup_{k \geq 1} \Phi (s_1(x), s_2(x), \ldots, s_k(x), 0, 0 , \ldots) \in [0,\infty].   \] 
 It turns out that 
\[  \mathfrak{S}_{\Phi}:= \{  \,   x \in \k    \, : \, \|x \|_{\Phi} < \infty   \, \} \]
and the $\|\, . \ \|_{\Phi}$-closure in $\mathfrak{S}_{\Phi}$  of the  finite rank operators, that is
$$  \mathfrak{S}_{\Phi}^{(0)}:= \overline{\mathfrak{F}\, \,  }^{\| \, . \, \|_{\Phi}} , $$
are symmetrically-normed ideals. It is not difficult to show that  $\mathfrak{S}_{\Phi}^{(0)}=\mathfrak{S}_{\Phi}$ if and only if $\mathfrak{S}_{\Phi}$ is separable. Moreover, any separable symmetrically-normed ideal coincides with  some $\mathfrak{S}_{\Phi}^{(0)}$ (see \cite[p. 89]{GK}).


\medskip

\noi \textbf{Pinching operators.}  
Let $\Phi$ a symmetric norming function, and $\I=\mathfrak{S}_{\Phi}$.   Recall  that given a family $\{ \, p_i \, \}_1 ^w$ ($1 \leq w \leq \infty$) of mutually orthogonal hermitian projections, i.e. 
\[ p_i=p_i^*, \, \, \, \, \, \, \, \, \, \, \, p_ip_j=\delta_{ij},    \]
we define the pinching operator associated with the family by 
$$P: \I \longrightarrow \I, \, \, \, \, \, \, \, \,  P(x)=\sum_{i=1}^{w} p_i x p_i, $$ 
 Notice that we might have $w=\infty$.  Since $x $ is compact,  the series,  which at first converges in the strong operator topology, turns out to be convergent in the uniform norm. It is also  noteworthy that $P$ is well defined in the sense that  $P(x) \in \I$ whenever $x\in \I$ (see \cite[p. 82]{GK}). 




Bellow we need to consider   the Banach algebra $\BI$  of all bounded operators on  $\I$   with the usual operator norm: for $X \in \BI$, 
\[  \|  X \|_{\BI}=\sup_{\|y\|_{\I}=1} \| X (y) \|_{\I}.    \]
We collect  some basic properties of pinching operators in the next proposition. 
 
\begin{proposition}
Let $\Phi$ a symmetric norming function, and $\I=\mathfrak{S}_{\Phi}$. Let $P$ be the pinching operator associated with a family $\{ \, p_i \,\}_{1}^w$. The following assertions hold:
\begin{enumerate}
\item[i)] $P^2=P$. 
\item[ii)] $P$ is a module map over its range.
\item[iii)] $P(x)^*=P(x^*)$.
\item[iv)] $P$ is continuous. In fact, $\| P \|_{\BI}=1$.
\end{enumerate}
\end{proposition}
\begin{proof}
The proofs of $i)-iii)$ are trivial. For a proof of $iv)$ we refer the reader to \cite[p. 82]{GK}.
\end{proof}


\medskip


\noi \textbf{Submanifolds.}
In the paper we will use  different notions of submanifold of a (Banach) manifold. Since the terminology is not uniform in the literature, we need to mention that we follow Bourbaki \cite{Bour}.  To be precise, let $M$ be a  manifold and $N$ a topological space contained in $M$. Recall that a subspace $F$ of a Banach space $E$ is said to be complemented if $F$ is closed and there exists a closed subspace $F_1$ such that $F\oplus F_1=E$. We will use the following definitions: 
\begin{itemize}
	\item $N$ is a \textit{submanifold} of $M$ if for each point $x \in N$ there exists a Banach space $E$ and a chart $(\mathcal{W},\phi)$ at $x$, $\phi:\mathcal{W}\subseteq M \longrightarrow E$, such that $\phi(\mathcal{W}\cap N)$ is a neighborhood of $0$ in a complemented  subspace of $E$. 
	\item $N$ is a \textit{quasi submanifold} of $M$ if for each point $x \in N$ there exists a Banach space $E$ and a chart $(\mathcal{W},\phi)$ at $x$, $\phi:\mathcal{W}\subseteq M \longrightarrow E$, such that $\phi(\mathcal{W}\cap N)$ is a neighborhood of $0$ in a closed subspace of $E$.
\end{itemize}

The following criterion will be useful (see \cite{Bour}). 

\begin{proposition}\label{sous varietes}
Let $M$ be a manifold, $N$ be a topological space and $N\subseteq M$. Then $N$ is a submanifold (resp. quasi submanifold) of $M$   if and only if the topology of $N$ coincides with the  topology  inherited from $M$ and the differential map of  the inclusion map $N \hookrightarrow M$ has complemented range (resp. closed range) at every $x \in N$.  
\end{proposition}

\section{Differential structure of $\UI(P)$}

\noi  Throughout this section, let $\Phi$ a symmetric norming function, and $\I=\mathfrak{S}_{\Phi}$.  Let  $P$ be the pinching operator associated with a family of mutually orthogonal projections $\{ \, p_i \,\}_{1}^w$ $(1\leq w \leq \infty)$. 
We first show that $\UI(P)$  has a smooth  manifold structure endowed with the quotient topology.

\begin{lemma}\label{calc isotropia}
Let $x \in \B$. Then $L_xP=PL_x$ if and only if $x=\sum_{i=0}^{w}p_i x p_i$. 
\end{lemma}
\begin{proof}
Suppose that  $L_xP=PL_x$, which actually means that 
\[  \sum_{i=1}^{w} (p_i x - xp_i)yp_i=0, \]
for all $y \in \I$. Let  $i \geq 0$ and $(e_n)_n$ be a sequence of finite rank projections such that $e_n \leq p_i$ and $e_n \nearrow p_i$ in the strong operator topology. We first assume that $i\geq 1$. Replacing $y$ by $e_n$, we get $p_ixe_n=xe_n$ for all $n \geq 1$. This gives $p_ixp_i=xp_i$ for all $i\geq 1$. Thus $p_jxp_i=0$ for all $i \geq 1$ and $j \geq 0$. 
In the case in which $i=0$ we replace $y$ by $p_0x^*$. Then we see that $p_ixp_0=0$ for all $i \geq 1$, and thus $x$ must be block diagonal. The proof of the converse assertion is trivial.   
\end{proof}


\begin{proposition}\label{est homog}
Let $\Phi$ a symmetric norming function, and $\I=\mathfrak{S}_{\Phi}$. Then  $\UI(P)$ is a real analytic homogeneous space of $\UI$.
\end{proposition}
\begin{proof}
Note that the isotropy group at $P$ of the natural underlying action of $\UI$ is 
\[ G=\{ \, u \in \UI \, : \, L_u P = P L_u \, \}.  \]
It is a closed subgroup of $\UI$. Its Lie algebra can be identified with
\[ \g=\{ \, z \in \I_{sh} \, : \, L_zP=PL_z \, \}.   \]
We will prove that $G$ is a Banach-Lie subgroup of $\UI$. Let $u=e^z \in G$, with $z \in \I_{sh}$ and $\|z\|_{\I}< \pi$. By the condition on the norm of $z$,  we have $z=log(u)=\sum_{n=0}^{\infty} \frac{(-1)^n}{(n+1)}(u-1)^{n+1}$. Notice that $L_uP=PL_u$, or $L_{u-1}P=PL_{u-1}$, clearly implies $L_{r(u-1)}P=PL_{r(u-1)}$ for any polynomial $r \in \RR[X]$, and by continuity we have $L_zP=PL_z$. Denote by $\exp_{\UI}: \I_{sh} \longrightarrow \UI$, $\exp_{\UI} (z)=e^z$ the exponential map of the Banach-Lie group $\UI$. Hence we have proved that $\exp_{\UI}( \g \cap V)=G \cap \exp_{\UI} (V)$, for any sufficiently small neighborhood $V$ of the origin in $\I_{sh}$.   

On the other hand, by  Lemma \ref{calc isotropia} we can rewrite the Lie algebra as
\[  \g=\bigg\{  \, \sum_{i=0}^{w}p_izp_i  \, : \,     z \in \I_{sh}          \, \bigg\}, \]
which is a real closed subspace of $\I_{sh}$. Moreover, the following  subspace
\[ \mathcal{M}=\{ \, z \in \I_{sh} \, : \, p_izp_i=0, \, \forall \, i \geq 0  \,   \} =\bigg\{  \, \sum_{i\neq j}p_izp_j  \, : \, z \in \I_{sh} \, \bigg\}  \]
is a closed supplement for $\g$ in $\I_{sh}$. Then, $G$ is a Banach-Lie subgroup of $\UI$, and by   \cite[Theorem 8.19]{Up} we conclude that $\UI(P)$ is a real analytic homogeneous space of $\UI$.       
\end{proof}


\subsection{When  is  $\UI(P)$ a submanifold of $\BI$? The case $\I \neq \k$ }

\noi  In this section, we discuss the submanifold structure of $\UIP$  under the assumption that $\I \neq \k$. 
 Recall that given the pinching operator $P$ associated with a  family of mutually orthogonal projections $\{ \, p_i \, \}_1 ^w$ $(1\leq w \leq \infty)$, we may   consider the larger family   $\{ \, p_i \, \}_0 ^w$, where $p_0=1-\sum_{i=1}^w p_i$.  However, the pinching operator $P$ is always associated with the first family  $\{ \, p_i \, \}_1 ^w$. The following estimate will be  useful.

\begin{lemma}\label{desigualdad} 
Let $\Phi$ a symmetric norming function, and $\I=\mathfrak{S}_{\Phi}$.  Then
$$\| L_x P -PL_x \|_{\BI} \geq \| p_i xp_j\|,$$ 
for  $x \in \I$, $i\geq 1$, $j\geq 0$ and $i \neq j$. 	 
\end{lemma}
\begin{proof}
  Consider  the Schmidt expansion of the compact operator $p_i x p_j$, namely
\[ p_ixp_j = \sum_{k=1}^{\infty} s_k \, \xi_k \otimes \eta_k , \]
where $s_k$ are the singular values of $p_i x p_j$ and $(\xi_k)_k$, $(\eta_k)_k$ are orthogonal systems of vectors. In particular, there is a vector $\xi \in R(p_j)$, $\|\xi \|=1$, such that $p_ixp_j\xi=\| p_i x p_j\| \xi$. Pick  any $\eta \in R(p_i)$ such that $\|\eta\|=1$. Then  note that
\[  (L_xP - PL_x )(\xi \otimes \eta)=-p_i x (\xi \otimes \eta) p_i = - p_i x p_j (\xi \otimes \eta) = - \| p_i x p_j \| \, \xi \otimes \eta.  \] 
We thus  get
\[
\| L_x P -PL_x \|_{\BI} \geq \| (L_x P -PL_x )(\xi \otimes \eta)\|_{\I} = \| p_i x p_j \|.
\]\end{proof}

\noi The first obstruction for $\UI(P)$ to be a submanifold of $\BI$ lies in the fact that its tangent spaces may not be closed. The tangent space of $\UI(P)$ at $Q$ (i.e. the derivatives at $Q$ of smooth curves inside $\UI(P)$) is apparently given by
\[  (T\UI(P))_Q= \{ \,   L_zQ-QL_z \, : \, z \in \I_{sh} \, \}. \]
We denote  tangent vectors briefly by $[L_z,Q]$.  



\begin{lemma}\label{tgte cerrado}
Assume that  $\I \neq \k$.
 Then tangent spaces of $\UIP$ are closed in $\BI$ if and only if $w < \infty$ and  there is only one infinite rank projection in the family $\{ \,  p_i \, \}_0 ^w$. 
\end{lemma}
\begin{proof}
It suffices to prove the statement for the tangent space  at $P$. Indeed, if $Q=L_uPL_{u^*}$ for some $u \in \UI$, then $[L_z , Q]=L_u[L_{u^*zu},P]L_{u^*}$. Thus $(T\UI(P))_Q$ is closed in $\BI$ if and only if $(T\UI(P))_P$ is closed in $\BI$.

Suppose that $(T\UI(P))_P$ is closed in $\BI$.  Let $x \notin \I$ be a  compact operator and  $(e_n)_n$ be a sequence of    finite rank projections such that $e_n \nearrow 1$ in the strong operator topology. Since $x$ is compact, the sequence of finite rank operators $z_n=e_nxe_n$   satisfies $\|x-z_{n}\| \rightarrow 0$. Let $\Re e (\, . \,)$ be the real part of an operator, then 
\begin{align*}
\| \, [L_{\Re e(z_{n})}, P]  -	[L_{\Re e(x)},P] \, \|_{\BI} & 
	\leq 2 \| L_{\Re e(z_n)} - L_{\Re e(x)} \|_{\BI} \\
	& = 2 \| \Re e(z_n) - \Re e(x) \| \leq 2 \|z_n - x \|	\rightarrow 0.
\end{align*}	 
Then there exists some $z_0 \in \I_{sh}$ such that 
$[L_{z_0},P]=[L_{i\Re e(x)},P]$. We can proceed analogously with the imaginary part $\Im m (\, . \,)$ to find another operator $z_1 \in \I_{sh}$ such that $[L_{z_1},P]=[L_{i\Im m(x)},P]$. Hence we obtain $[L_x,P]=[L_{z},P]$ for $z=-iz_0 + z_1 \in \I$. By Lemma \ref{calc isotropia} the latter can be  rephrased as  
\[  x-z= \sum_{i=0}^{w} p_i (x-z)p_i. \]
In particular, we see that 
\begin{equation}\label{contradiccion}
x-\sum_{i=0}^{w} p_i xp_i  \in \I.
\end{equation}  
 Recall that  $\I=\mathfrak{S}_{\Phi}$ for some  symmetric norming function $\Phi$.  Since $\I$ is different from the compact operators, there exists a sequence of positive numbers $(a_n)_n$ such that $a_n \rightarrow 0$ and $\Phi((a_n)_n)=\infty$. 

Suppose that the family  $\{ \, p_i \, \}_{0}^w$ has two  projections $p_i$, $p_j$, $i\neq j$, such that both have infinite rank. Let $(  \xi_n  )_{n}$ be an orthonormal basis of $R(p_i)$ and $(  \eta_n  )_{n}$ be an orthonormal basis of $R(p_j)$. Consider the following compact operator:
\[  x=\sum_{n=1}^{\infty} a_n \, \xi_n \otimes \eta_n. \]   
From our choice of the sequence $(a_n)_n$ it follows that $x \notin \I$. Thus  we find that $x=p_ixp_j=x-\sum_{i=0}^{w}p_ixp_i \notin \I$, which contradicts equation (\ref{contradiccion}). Hence it is impossible to have two different projections with infinite rank in the family $\{ \, p_i \, \}_{0} ^w$.

It remains to prove that $w < \infty$. Suppose that there is an  infinite number of  projections $p_{1}, p_{2}, \ldots$.  We can construct an orthonormal system of vectors $( \xi_i  )_{i}$ such that $\xi_i \in R(p_{i})$. Then we define the following compact operator:
\[  x= \sum_{n=1}^{\infty} a_n \, \xi_{n+1} \otimes \xi_n. \]  
It is easily seen  that $x=\sum_{n=1}^{\infty} p_{n+1}xp_n=x-\sum_{i=0}^{\infty}p_ixp_i \notin \I$. We thus get again a contradiction with equation (\ref{contradiccion}).

\smallskip

In order to prove the converse we assume that the  family $\{ \, p_i  \, \}_0 ^w$ satisfies  $w < \infty$ and it has only one projection $p_{i_0}$ with infinite rank.   Let $(z_{k})_{k}$ be a sequence in $\I_{sh}$
such that $\|\,[L_{z_{k}},P]\, - X\|_{\BI} \rightarrow 0$, where $X \in \BI$. It is worth  noting that by Lemma \ref{calc isotropia} the sequence $(z_{k})_{k}$ can be chosen satisfying $p_iz_kp_i=0$  for all $k$ and $i=0, \ldots, w$.  Since $(\,[L_{z_k},P]\,)_k$ is a Cauchy sequence in $\BI$,   Lemma \ref{desigualdad} implies that    
\[\|p_{i}(z_{k}-z_{r})p_{j}\|
	\underset{k,r \rightarrow \infty}{\longrightarrow} 0
\]
for $i=1, \ldots , w$, $j=0, \ldots, w$ and $i \neq j$.  Note that the rank of the operators $p_i(z_k -z_r)p_j$ is uniformly bounded on the subscripts $k$ and $r$ by $C:= \max \{  \, \rank (p_j)  \, : \, j=0, \ldots, w, \, j \neq i_0  \, \}$. Then we get
	\[\|p_{j}(z_{k}-z_{r})p_{i}\|_{\I} \leq 	C \|p_{j}(z_{r}-z_{k})p_{i}\|
	\underset{k,r \rightarrow \infty}{\longrightarrow} 0.
\]
Hence each $(p_{j}z_k p_{i})_k$ converges in the ideal norm to some $z_{ij} \in \I$.  We can construct an operator $z$ by defining its matricial blocks with respect to the projections $p_0, p_1, \ldots , p_w$ as follows:  
$$ p_izp_j :=\left\{
\begin{array}{cc}
0              & \, \, \, \, \, \, \, \, \, \text{if} \, \, \, \, \, \,  i=j,\\
z_{ij} & \, \, \, \, \, \, \, \, \, \text{if}  \, \, \, \, \, \, i \neq j.
\end{array}\right.
$$
Then $z$ is a skew-hermitian operator in $\I$ satisfying 
	\[\|z - z_{k}\|_{\I} \leq	\d\sum_{i \neq j} \|p_{j}zp_{i}-p_{j}z_{k}p_{i}\|_{\I} = \sum_{i \neq j} \| z_{ij} - p_jz_kp_i \|_{\I} \rightarrow 0.
\]
Therefore
	\[\| \, [L_{z_{k}},P] -
	 [L_{z},P] \,\|_{\BI} \leq 2 \|L_{z_k} - L_{z}\|_{\BI} =2 \| z_k - z\| \leq 2 \| z_k - z \|_{\I} \rightarrow 0.  
\]
Hence we conclude $X=[L_z ,P]$, and the lemma is proved.
\end{proof}

\noi We can endow $\UI(P)$ with two natural topologies. According to Proposition  \ref{est homog} we have that $\UI(P)\simeq \UI / G$ has a real analytic manifold structure in the quotient topology  in such way that the map $\pi: \UI \longrightarrow \UI(P)$, $\pi(u)=L_uPL_{u^*}$ is a real analytic submersion. On the other hand, we can regard  $\UI(P)$ as a subset of $\BI$ with the inherited topology. In this case, we denote the projection map by $\tilde{\pi}: \UI \longrightarrow \UI(P)$, $\tilde{\pi}(u)=L_uPL_{u^*}$. Note that $\tilde{\pi}$ is also continuous, and the following diagram commutes

\begin{displaymath}
\xymatrix{
\UI \ar[r]^{\pi} \ar[dr]_{\tilde{\pi}} & \UI(P) \ar[d]^{id}  \hspace{-0.5cm}& \hspace{-.5cm} \simeq \UI /G  \\
                                   & \UI(P)             \hspace{-0.5cm} &  \hspace{-0.5cm} \subseteq \BI 
}
\end{displaymath}

\smallskip

\noi Here $id$ stands for the identity map. Note that $id$ is always continuous, but it may not be a homeomorphism. In fact, we will show that the two topologies defined on $\UI(P)$ coincide if and only if  tangent spaces are closed. As we will see, the proof of this result depends on the existence of continuous local cross sections for the action.

\medskip

\begin{remarknf}Let $P$ be the pinching operator associated with a family  $\{ \, p_i \}_1 ^w$. We will consider the unitary orbit of each projection $p_i$, i.e.   \[   \o _{i} := \{ \, up_iu^* \, : \, u  \in \UI \,   \}. \]
If $\I$ is the ideal of Hilbert-Schmidt operators, the above defined orbits are usually known as the connected component of $p_i$ in the restricted Grassmannian (see e.g. \cite{ps}). Note that $\o_i \subseteq p_i + \I$, so we may endow  each orbit  with the subspace topology defined by the metric $(up_iu^* , vp_iv^*)\mapsto \|  up_iu^*  - vp_iv^* \|_{\I}$. 
\end{remarknf}

\begin{lemma}\label{cont fes}
Assume that   $w < \infty$ and there is only one infinite rank projection  in  the family $\{ \, p_i \}_0 ^w$. Then the map $$F_i:  \UI(P)   \longrightarrow \mathcal{O}_i, \, \, \, \, \, \, \, \,  F_i(L_u P L_{u^*})=up_iu^*$$ is continuous for $i=0, 1,\ldots w$, when $\UI(P)$ is endowed with the topology inherited from $\BI$.
\end{lemma}
\begin{proof}
We first show that the function $F_i$ is well defined for $i=0,1, \ldots ,w$. From Lemma \ref{calc isotropia} we know that $L_uPL_{u^*}=L_vPL_{v^*}$ implies $v^*u=\sum_{i=0}^w p_iv^*up_i$. Then we get $v^*up_i=p_iv^*up_i=p_iv^*u$, or equivalently, $up_iu^*=vp_iv^*$. 

To prove  the continuity of $F_i$ we will actually  see that $F_i$ is Lipschitz. Since the underlying actions are isometric, it suffices to estimate the distance from $F_i(L_uPL_{u^*})=up_iu^*$ to  $F_i(P)=p_i$.  
For $u \in \UI$, set $a(u):= \| L_uPL_{u^*} - P \|_{\BI}= \| \, [L_u ,P] \, \|_{\BI} $.    From Lemma \ref{desigualdad} it follows that
\[  \| p_i u p_j \| = \| p_i (u-1) p_j \| \leq a(u), \] 
for $j=0,1, \ldots, w$, $i=1, \ldots ,w$ and $i\neq j$. The same estimate can be extended for all $i \neq j$. In fact, we have
\[  \|p_jup_i \|= \| p_i u^* p_j \| \leq a(u^*)= a(u).   \]  
Let $p_{i_0}$ the unique infinite rank projection in the family $\{ \, p_i \, \}_0 ^w$. For $u \in \UI$, we note that $\rank(p_iup_j) \leq  \min \{ \, \rank(p_i) \, , \, \rank(p_j) \, \}$, and then  we get
$$ 
  \max \{ \, \rank(p_{j}up_i) \, : \, i,j=0,1, \ldots, w, \, i \neq j  \,  \} \leq \max \{ \, \rank(p_{j}) \, : \, j=0,1, \ldots, w, \, j \neq i_0 \,  \}:=C .
$$
This implies that 
\[  \|  p_i u p_j \|_{\I} \leq C \|  p_i u p_j \| \]
for $i \neq j$. Thus we get
\begin{align}\label{dominio bueno}
\| F_i(L_uPL_{u^*}) - F_i (P) \|_{\I} & = \| up_i - p_i u \|_{\I} 
\leq \sum_{j: j \neq i} \|  p_j u p_i \|_{\I} + \sum_{j: j \neq i} \|  p_i u p_j \|_{\I} \nonumber \\
& \leq C  \bigg( \sum_{j: j \neq i}  \|  p_j u p_i \| + \sum_{j: j \neq i} \|  p_i u p_j \| \bigg) \nonumber \\ 
& \leq 2 w C \| L_uPL_{u^*} - P \|_{\BI},
\end{align}
which shows that $F$ is Lipschitz.
\end{proof}

\begin{remarknf}\label{sequence prop}Let $\mathcal{M}$ be the supplement of the Lie algebra defined in Proposition \ref{est homog}. Suppose that $w=\infty$ or there exist two different infinite rank projections in the family $\{ \, p_i \, \}_0 ^w$. Under the  assumption that $\I \neq \k$, we will  construct a sequence $(z_k)_k$ in   $\mathcal{M}$ satisfying $\|z_k\| \to 0$ and $\| z_k \|_{\I}=1$. To this end, put
\[ a_k := \Phi(\underbrace{1,1,\ldots , 1}_{k}, 0 ,0 , \ldots ),        \]
where $\Phi$ is a symmetric norming function such that  $\I=\mathfrak{S}_{\Phi}$. Since $\I\neq \k$,  it follows that $\Phi$ is not equivalent to the uniform norm of $\ell^{\infty}$, so that $a_k \to \infty$  (see \cite[p. 76]{GK}). 
In the case in which $w=\infty$, let $(\xi_i)_i$ be an orthonormal system such that $\xi_i \in R(p_i)$ for all $i\geq 1$. It is not difficult to see that the sequence defined by
\[  z_k:= a_{2k}^{-1} \, \sum_{i=1}^{k} \, \xi_{2i -1} \otimes \xi_{2i} -  \xi_{2i} \otimes \xi_{2i-1}    \]
satisfies the required properties. In the case in which  there exist two different infinite rank projections $p_i$ and $p_j$, let $(\xi_i)_i$ be an orthonormal system such that $\xi_{2k-1} \in R(p_i)$ and $\xi_{2k} \in R(p_j)$ for all $k\geq 1$. Then we can define the sequence $(z_k)_k$ in the same fashion as before.\end{remarknf}

\begin{lemma}\label{topology}
Assume that $\I \neq \k$.  Then the following conditions are equivalent: 
\begin{itemize}
\item[i)]  The quotient topology of $\UI(P)$ coincides with the topology inherited from $\BI$.  
\item[ii)] $w < \infty$ and there is only one infinite rank projection  in  the family $\{ \, p_i \}_0 ^w$.
\end{itemize}
\end{lemma}
\begin{proof}
Suppose that the quotient topology of $\UI(P)\simeq \UI /G$ coincides with the topology inherited form $\BI$. Let $\mathcal{M}$ be the supplement of the Lie algebra of $G$ defined in Proposition \ref{est homog}. Recall that a real analytic atlas of $\UI(P)$ compatible with the quotient topology can be constructed by translation of the homeomorphism
\[  \psi: \mathcal{W} \subseteq \mathcal{M} \longrightarrow \psi(\mathcal{W}),  \, \, \, \, \, \, \, \psi(z)=(\pi \circ \exp_{\UI})(z)=  L_{e^z} P L_{e^{-z}},   \]
where $\mathcal{W}$ is an open neighborhood of $0 \in \mathcal{M}$ and $\psi(\mathcal{W})$ an open neighborhood of $P$ (see for instance \cite[Theorem 4.19]{B}). 
Assume that the family $\{ \, p_i \, \}_0 ^w$ does not satisfy the claimed properties. This leads us to consider two cases, namely $w=\infty$ or there exist two different infinite rank projections in $\{ \, p_i \, \}_0 ^w$. In any case we can find a sequence $(z_k)_k$ in $\mathcal{M}$ such that $\| z_k\|\to 0$ and $\|z_k\|_{\I}=1$ according to Remark \ref{sequence prop}. Then note that
\[ \| L_{e^{z_k}} P L_{e^{-z_k}} - P \|_{\BI} = \| \, [ L_{e^{z_k}- 1} , P ] \,\|_{\BI} \leq 2 \| e^{z_k} - 1 \| \to 0,  \]
and using that the quotient topology of $\UI(P)$ coincides with the subspace topology, we arrive at a contradiction:  $\| z_k\|_{\I}=\| \psi^{-1}( L_{e^{z_k}} P L_{e^{-z_k}})\|_{\I} \to 0$.

\smallskip

To prove the converse, assume that  $w< \infty$ and there is only one infinite rank projection  in  the family $\{ \, p_i \}_0 ^w$. Clearly, our assertion about the topology of $\UI(P)$ will follow if we show that the projection map
\[ \tilde{\pi} : \UI \longrightarrow \UI(P), \, \, \, \, \, \, \tilde{\pi}(u)=L_u P L_{u^*} \]  
have continuous local cross sections, when $\UI(P)$ is considered with the relative topology of $\BI$. To this end, for $i=0,1, \ldots, w$, we need to consider the orbits  
\[   \o _{i} := \{ \, up_iu^* \, : \, u  \in \UI \,   \}. \]
In  \cite[Proposition 2.2]{AL8} the authors showed that the maps
\[  \pi_i: \UI \longrightarrow \o _i,  \, \, \, \, \,  \pi_i(u)=up_iu^*,\]
has continuous local cross sections, when $\I$ is the ideal of Hilbert-Schmidt operators. Actually, the same proof works out  for any symmetrically-normed ideal $\I$, so we have that there exist continuous maps 
\[  \psi_i : \{ \, q \in \o_i    \,: \,  \| q - p_i \|_{\I} < 1    \,   \} \subseteq p_i + \I \longrightarrow \UI \]
such that $\psi_i(up_iu^*)p_i\psi_i(up_iu^*)^*=up_iu^*$ for any $u \in \UI$ such that $\| up_iu^*- p_i \|_{\I}< 1$.

Now we can  explicitly give  the required section for  $\tilde{\pi}$, namely
\[ \sigma: \big\{  \,  Q \in \UI(P)  \, : \,   \| Q - P \|_{\BI} < 1/2wC   \,  \big\} \longrightarrow \UI, \, \, \, \, \, \,  \sigma (L_uPL_{u^*})=\sum_{i=0}^w \psi_i(up_iu^*)p_i. \]    
If $Q=L_uPL_{u^*}$ lies in the domain of $\sigma$, then by the estimate (\ref{dominio bueno}) in Lemma \ref{cont fes}, the operators $up_iu^*$ do lie in the domain of each $\psi_i$.   Our next task is to show that $\sigma=\sigma(L_uPL_{u^*}) \in \UI$. In fact, we see that
\[  \sigma \sigma^*= \bigg(\sum_{i=0}^w \psi_i(up_iu^*)p_i \bigg) \bigg(\sum_{i=0}^w p_i\psi_i(up_iu^*)^*\bigg) = \sum_{i=0}^w \psi_i(up_iu^*)p_i\psi(up_iu^*)^*=\sum_{i=0}^w up_iu^*=1.  \] 
Note that $p_j\psi_j(up_ju^*)^* \psi_i(up_iu^*)p_i= \psi_j(up_ju^*)^* up_jp_iu^* \psi_i(up_iu^*)=\delta_{ij}$, then
\[ \sigma^*\sigma=  \bigg(\sum_{i=0}^w p_i\psi_i(up_iu^*)^*\bigg) \bigg(\sum_{i=0}^w \psi_i(up_iu^*)p_i \bigg)= \sum_{i=0}^w p_i = 1.    \]
Also  we see that $$  \sigma -1 = \sum_{i=0}^w (\psi_i(up_iu^*) -1)p_i  \in \I.$$
On the other hand, the map $\sigma$ is  actually  a section for $\pi$: for any $y \in \I$,  
\begin{align*}
L_{\sigma(L_uPL_{u^*})}P L_{\sigma(L_uPL_{u^*})^*}(y) & =\sum_{i=0}^w \sigma(L_uPL_{u^*})p_i\sigma(L_uPL_{u^*})^*yp_i \\
& = \sum_{i=0}^w \psi_i(up_iu^*)p_i \psi_i(up_iu^*)^*yp_i \\
& =   \sum_{i=0}^w up_iu^* yp_i= L_uPL_{u^*}(y).
\end{align*}
 Finally, to show   the continuity of $\sigma$, it is enough to remark that 
\[  \sigma (L_uPL_{u^*})=\sum_{i=0}^w \psi_i( F_i(L_uPL_{u^*})) p_i  \]
and use the continuity of each $F_i$, which has  already been proved in Lemma \ref{cont fes}. 
\end{proof}

\noi Now our main result on the differential structure of $\UI(P)$ follows.

\begin{theorem}\label{complemento}
Let $\Phi$ a symmetric norming function, and $\I=\mathfrak{S}_{\Phi}$. Assume that $\I \neq \k$. Let $P$ be the pinching operator associated with a family  $\{ \, p_i \}_1 ^w$ $(1\leq w \leq \infty)$.
Then  the following assertions are equivalent:
\begin{itemize}
\item[i)] The quotient topology on $\UI(P)$ coincides with topology inherited from $\BI$.
\item[ii)] Tangent spaces of $\UI(P)$ are closed in $\BI$.
\item[iii)]    $w < \infty$ and  there is only one infinite rank projection in the family $\{ \,  p_i \,\}_0 ^w$. 
\item[iv)] $\UI(P)$ is  a   submanifold of $\BI$. 
\end{itemize}
\end{theorem}
\begin{proof}
Suppose that $\UI(P)$ is a  submanifold of $\BI$. By Proposition \ref{sous varietes}, tangent spaces of $\UIP$ has to be closed in $\BI$. From Lemma \ref{tgte cerrado} it follows that the family $\{ \,  p_i \, \}_0 ^w$ satisfies the stated properties.

 Now we assume that  $w < \infty$ and  there is only one infinite rank projection in the family $\{ \,  p_i \, \}_0 ^w$. According to  Lemma \ref{tgte cerrado} and  Lemma \ref{topology}, what is left to prove is that  tangent spaces are complemented in   $\BI$. Clearly, it suffices  to show that $(T\UIP)_P$ is complemented in $\BI$.

  We will divide the proof into two cases according to whether the rank of $p_0$ is infinite or finite. Let us first assume that  $\rank(p_0)=\infty$, so that $\rank(p_i)<\infty$ for all $i=1, \ldots , w$. Then $X(p_i)$ is well defined for any $X \in \BI$, $i=1, \ldots, w$, and we can set
\[ \hat{z}: \BI \longrightarrow \I_{sh}, \, \, \, \, \, \, \hat{z}(X)= 2 i \Im m \bigg( \sum_{i=1}^{w} \sum_{j=0}^{i-1} p_j X(p_i) \bigg).  \]
Clearly $\hat{z}$ is a continuous linear operator. Then we define a  bounded linear  projection onto the tangent space by $$E: \BI \longrightarrow (T\UI(P))_P, \, \, \, \, \, \, \, \,  E(X)=[L_{\hat{z}(X)}  ,P].$$ 
In order to show that $E$ actually defines a projection we pick  $X=[L_z,P]$ for some $z \in \I_{sh}$. Notice that $X(p_i)=(1-p_i)zp_i$,  for all $i=1, \ldots, w$, then  we get that 
\[ \hat{z}(X)=  2 i \Im m \bigg(\sum_{i=1}^{w} \sum_{j=0}^{i-1} p_j z p_i \bigg)= z - \sum_{i=0}^np_izp_i.  \]  
From Lemma \ref{calc isotropia} we deduce that $E(X)=[L_{\hat{z}(X)},P]=X$, which proves that $E$ is a projection.  Finally,  the continuity of $\hat{z}$ easily implies that of $E$.

Now we consider the case in which the infinite rank projection is not $p_0$. Without loss of generality we may assume that $\rank(p_1)=\infty$.   Let us point out that the above definition of the operator $\hat{z}(X)$ does not work  in this case for two different reasons: on one hand,  since $p_1 \notin \I$ we cannot evaluate any $X \in \BI$ at $p_1$, and on the other hand,  every tangent vector $[L_z,P]$ vanishes at $p_0$. 

In order to solve this case we need to modify the definition of the operator $\hat{z}$. Recall that $\rank(p_0)<\infty$ since $\rank(p_1)=\infty$. 
Let $\eta_1 , \ldots , \eta_m$ be an orthonormal basis of $R(p_0)$. Let $\xi \in R(p_1)$ be a unit vector.   Then we define
\[ \hat{z}:\BI \longrightarrow \I_{sh}, \, \, \, \, \, \, \hat{z}(X)= 2i \Im m\bigg( \sum_{i=2}^{w} \sum_{j=0}^{i-1} p_j X(p_i)  -  \sum_{k=1}^m X(\eta_k \otimes \xi) \xi \otimes \eta_k   \bigg), \]   
and the projection onto the tangent space is 
$$E: \BI \longrightarrow (T\UI(P))_P, \, \, \, \, \, \, \, \,  E(X)=[L_{\hat{z}(X)}  ,P].$$ 
It is apparent that $E$ is continuous, so we are left with the task of proving that $E$ is a projection. To this end, let $X=[L_z,P]$ for some $z \in \I_{sh}$. Note that
\[  X(\eta_k \otimes \xi)=\sum_{i=1}^w (zp_i- p_iz)(\eta_k \otimes \xi)p_i=(zp_1 -p_1z)(\eta_k \otimes \xi)p_1=-p_1z(\eta_k \otimes \xi),\]
and then
\[  \sum_{k=1}^m X(\eta_k \otimes \xi)\xi \otimes \eta_k =-p_1zp_0.\]
Thus we get
\[ \hat{z}(X)= 2 i \Im m \bigg(   \sum_{i=2}^{w} \sum_{j=0}^{i-1} p_j z p_i  + p_1zp_0  \bigg)=z-\sum_{i=0}^np_izp_i. \]
Hence we conclude that $E([L_z,P])=[L_z,P]$, and the proof is complete.
\end{proof}

\subsection{When  is  $\UK(P)$ a submanifold of $\BK$? }

In this section we turn to the case $\I=\k$. The following estimate is a somewhat improved version of Lemma \ref{desigualdad}.

\begin{lemma}\label{des ctos}
Let $x \in \k$ such that $p_ixp_i=0$ for all $i \geq 1$. Then
\[   \| L_xP-PL_x\|_{\BK} \geq \| x(1-p_0) \|, \]
where $p_0=1-\sum_{i=1}^w p_i$.
\end{lemma}
\begin{proof}
To estimate the norm of $L_xP-PL_x$ as an operator acting on $\k$ we  need to consider the following projections: if $\rank(p_i)=\infty$, let $(p_{i,k})_{k}$ be a sequence of finite rank projections satisfying $p_{i,k} \leq p_i$ and $p_{i,k} \nearrow p_i$, and if $\rank(p_i)<\infty$,  we set $p_{i,k}=p_i$ for all $k \geq 1$. Now assume that the pinching operator $P$ is associated with a family $\{ \, p_i  \, \}_1 ^w$ such that $w < \infty$. Then the projections given by $e_k=\sum_{i=1}^n p_{i,k}$ have finite rank. We thus get
  \begin{align*}
\| L_x P -PL_x\|_{\BK} & \geq \|  (L_x P -PL_x )(e_k) \| 
 = \bigg\| \sum_{i=1}^w (1-p_i)xp_{i,k} \bigg\| 
= \bigg\| x \sum_{i=1}^w p_{i,k} \bigg\|,
\end{align*}
where in the last equality we use that $p_ixp_i=0$. Using that $ x \in \k$ and $p_{i,k} \nearrow p_i$, we find that
\[ \| L_x P -PL_x\|_{\BK} \geq \| x(1-p_0) \|. \]  
In the case where $w=\infty$, we set $e_{n,k}=\sum_{i=1}^n p_{i,k}$. In the same fashion as above we find that
\[  \| L_x P -PL_x\|_{\BK}  \geq  \bigg\| x \sum_{i=1}^n p_{i,k} \bigg\| \] 
Letting $k \to \infty$, we have 
\[ \| L_x P -PL_x\|_{\BK}  \geq  \bigg\| x \sum_{i=1}^n p_{i}\bigg\|, \]
for all $n \geq 1$. Now letting $n \to \infty$,  we get the estimate in this case.  
\end{proof}


\begin{proposition}\label{tengent}
Tangent spaces of $\UK(P)$ are closed in $\BK$.
\end{proposition}
\begin{proof}
By the remark at the beginning of the proof of Lemma \ref{tgte cerrado},  we may restrict, without loss of generality, to verify the statement for the tangent space at $P$. Let $(z_k)_k$ be a sequence in $\k_{sh}$ such that $p_iz_kp_i=0$ for all $i \geq 0$ and $k \geq 1$. Suppose that $ \| \, [ L_{z_k} , P ] - X \,  \|_{\BK} \rightarrow 0$ for some $X \in \BK$. According to Lemma \ref{des ctos}, 
\[  \|(z_k - z_r) (1-p_0) \| \leq \| \, [L_{z_k - z_r} ,P]  \, \|_{\BK}.     \]
Also note that 
\[  \| (z_k - z_r) p_0 \| = \| p_0(z_k - z_r)\| = \| p_0 (z_k - z_r) (1- p_0) \| \leq  \| \, [L_{z_k - z_r} ,P]  \, \|_{\BK}.         \] 
Therefore $(z_k)_k$ is a Cauchy sequence and thus has a limit $z_0 \in \k_{sh}$. Then we see that 
\[ \| \, [L_{z_k},P] - [L_{z_0},P]   \, \| \leq 2 \| z_k - z_0 \| \rightarrow 0.   \]
Thus we conclude that $X = [L_{z_0},P]$.
\end{proof}

\noi Now we turn to the study of the topology of $\UK(P)$. We will find that the quotient topology and the topology inherited from $\BK$ coincide regardless the number or rank of the projections in the family $\{  \,  p_i \, \}_0 ^w$.

\medskip

\begin{remarknf}\label{f0 cont}
Let $P$ be the pinching operator associated with a family $\{ \, p_i \, \}_1 ^w$ $(1 \leq i \leq w)$. In this subsection we need to   consider again the unitary orbit of the projections,  which we denote by 
\[   \o_i = \{ \, up_iu^* \, : \, u \in \UK \, \}. \] 
for $i=0, \ldots , w$.  We claim that the map
 \[ F_0: \UI(P) \longrightarrow \o_i, \, \, \, \, \, \, \, F_0(L_uPL_{u^*})=up_0u^*.   \]
is Lipschitz. In fact, according to Lemma \ref{des ctos} applied with $x=u-1 - \sum_{i=0}^wp_i(u-1)p_i=u - \sum_{i=0}^wp_iup_i$ we have that
\[ \| p_0 u (1-p_0)\|=   \bigg\| p_0\bigg(u - \sum_{i=0}^w p_iup_i\bigg)(1-p_0)  \bigg \| \leq \bigg \| \bigg(u - \sum_{i=0}^w p_iup_i \bigg)(1-p_0) \bigg\| \leq \| L_uPL_{u^*} - P \|_{\BK}\,.   \]
Replacing $u$ by $u^*$ we find that
\[  \| (1-p_0)up_0\|= \|p_0 u^* (1-p_0)\| \leq  \| L_uPL_{u^*} - P \|_{\BK}.   \]
Thus we get
\begin{align*} 
\| F_0(L_{u}PL_{u^*}) - F_0(P) \| & = \| up_0u^* - p_0 \|  \\
& \leq \| (1-p_0)up_0\| + \| p_0 u (1-p_0)\| \leq 2 \|L_uPL_{u^*}  - P \|, 
\end{align*}
which proves our claim.
\end{remarknf}

\medskip

\begin{lemma}\label{des top inf}
 Let $u,v \in \UK$. Then
\[   \bigg\| \sum_{i=0}^w up_iu^*p_i - vp_iv^*p_i   \bigg\| \leq 3 \|L_uPL_{u^*} - L_v P L_{v^*}  \|_{\BK},\]
where in  the case in which $w=\infty$ the series on the left side is convergent in the uniform norm.
\end{lemma}
\begin{proof}
For each $i \geq 1$, let $(p_{i,k})_k$ be a sequence of finite rank projections such that $p_{i,k}\leq p_i$ and $p_{i,k} \nearrow  p_i$. In case  $p_i$ has finite rank, we set $p_{i,k}=p_i$ for all $k$.   We will use the orthogonal projections defined by   $e_{k}=\sum_{i=1}^np_{i,k}$.  Put $a(u,v):=\|L_uPL_{u^*} - L_v P L_{v^*}  \|_{\BK}$. Then 
\[  \bigg\| \sum_{i=1}^w (up_iu^* - vp_iv^*)p_{i,k}   \bigg\| = \|(L_uPL_{u^*} - L_v P L_{v^*} )(e_{k})\|\leq a(u,v).     \]
Note that for each $i\geq1$, the operator $up_iu^* - vp_iv^*$ is compact. Letting $k \to \infty$, we get that
\[  \bigg\| \sum_{i=1}^w (up_iu^* - vp_iv^*)p_{i}   \bigg\| \leq a(u,v). \]
Combining this with the Remark \ref{f0 cont} it gives that
\begin{equation}\label{caso finito 1} 
 \bigg\| \sum_{i=0}^w (up_iu^* - vp_iv^*)p_{i}   \bigg\| \leq 3 a(u,v). 
\end{equation} 
This finishes the proof for the case  $w < \infty$. If $w = \infty$, we note that
\[  \sum_{i=0}^{\infty} up_iu^*p_i - vp_iv^*p_i= \sum_{i=0}^{\infty} u p_i(u^*-1)p_i    - vp_i(v^*-1)p_i + (u-v)p_i.    \]
Since the operators $u^*-1$, $v^*-1$ and $u-v$ are compact,  this series  converges in the uniform norm.  Letting $w \to \infty$ in (\ref{caso finito 1}),  the desired inequality follows.  
\end{proof}

\noi In the following proposition we extend the technique developed in \cite{AL8} to construct continuous local cross sections.

\begin{proposition}\label{topology ctoss}
The map 
\[ \pi: \UK \longrightarrow \UK(P) \subseteq \BK, \, \, \, \, \, \, \pi(u)=L_uPL_{u^*},  \]
has continuous local cross sections, when $\UK(P)$ is considered with the topology inherited from $\BK$.
\end{proposition}
\begin{proof}
Let $P$ be the pinching operator associated with a family $\{   \,  p_i \, \}_1 ^w$ $(1\leq w \leq \infty$).  Since the action of $\UK$ is isometric it will be enough to find a continuous section $ \sigma$  in a neighborhood of $P$. Also we will restrict ourselves to prove the case  $w=\infty$. The case $w<\infty$ needs less care, and it  can be handled  in much the same fashion.

 We consider the following neighborhood of $P$ to define the cross section, 
\[  \mathcal{V}:=\big\{  \,  Q \in \UK(P)  \, : \,   \| Q - P \|_{\BK} < 1/3   \,  \big\} . \]  
Given  $Q=L_uPL_{u^*} \in \mathcal{V}$, where $u \in \UK$, let $q_i=F_i(Q)=up_iu^*$  for $i \geq 0$. According to the proof of Lemma \ref{cont fes} the function $F_i$ is well defined.  Then,  we 
set 
$$ s=s(Q):= \sum_{i=0}^{\infty} q_i p_i \,.$$
This  series is convergent in the strong operator topology. In fact, we can rewrite the series as
\[  \sum_{i=0}^{\infty} q_i p_i= \sum_{i=0}^{\infty} up_i(u^*-1)p_i + (u-1)p_i + p_i ,\]
where the first and second summand on the right are convergent in the uniform norm, while the third is convergent in the strong operator topology.
 On the other hand, note that 
\[   \| s - 1\| \leq 3 \| Q - P \|_{\BK} < 1. \]
Then we get that $s$ is invertible. Moreover, it follows that 
\[  s-1 = u \bigg( \sum_{i=0}^{\infty} p_i (u^* -1 )p_i + 1 \bigg) - 1=u\sum_{i=0}^{\infty} p_i (u^* -1 )p_i + u-1 \in \k, \]
which is due to the fact that $\sum_{i=0}^{\infty} p_i (u^* -1 )p_i \in \k$. Now we will show that $$\sigma= \sigma(Q):=s|s|^{-1}$$ is a  continuous local cross section for $\pi$. To this end, note that $sp_i=q_ip_i=q_is$, so that $p_i |s|^{2}=s^*q_is=|s|^2p_i$, which   implies  
\[   \sigma p_i \sigma^{*}= s|s|^{-1} p_i |s|^{-1}s  =s p_i |s|^{-2} s^* =sp_is^{-1}=q_i. \]
This allows us to prove that $\sigma$ is a section: for any $y \in \k$, we have
\[  L_{\sigma} P L_{\sigma^*}(y)=\sum_{i=1} ^{\infty}  \sigma p_i \sigma^{*} yp_i = \sum_{i=1} ^{\infty}  q_i yp_i= Q(y).    \]
On the other hand, we have $|s|^2 -1 \in \k$, and consequently, $|s|-1=(|s|^2 - 1)(|s| + 1)^{-1} \in \k$. Therefore we can conclude 
\[  \sigma - 1 = s|s|^{-1} -1=(s -|s|)|s|^{-1}= (s-1)|s|^{-1} + (1-|s|)|s|^{-1} \in \k. \]
Hence  $\sigma \in \UK$. Let $Gl(\h)$ denote the group of invertible operators on $\h$. In order to prove the continuity of $\sigma$ we consider the subgroup of $Gl(\h)$ given by
\[    Gl_{\k}= \{ \,  g \in Gl(\h) \, : \, g -1 \in \k                    \,      \}.     \]
It is a Banach-Lie group endowed with the topology defined by $(g_1 , g_2) \mapsto \| g_1 - g_2\|$ (see \cite{B}). From Lemma \ref{des top inf} the map $s: \mathcal{V} \longrightarrow Gl_{\k}$ is continuous. Also note that the map  $Gl_{\k} \longrightarrow \UK$,  $s \mapsto s|s|^{-1}$, is real analytic by the regularity properties of the Riesz functional calculus. Thus $\sigma$ is continuous, being the composition of continuous maps.    
\end{proof}

Our next task in the study of the  submanifold structure of $\UK(P)$ is to ask about the existence of a supplement for $(T\UI(P))_P$ in $\BK$. The existence of such supplement is closely related to the fact that for an infinite dimensional Hilbert space $\h$ the compact operators  are not complemented in $\B$. A proof of this result can be found, for instance, in \cite{con}.  It is based on the following well known result:  $c_0$ (sequences which converges to zero) is not complemented in $\ell^{\infty}$ (bounded sequences). The reader can find a proof of this latter fact in \cite{w}.

\medskip

\begin{remarknf}\label{rem complemnt}
We will need a slightly modified version of the afore-mentioned result.  We first note that $\k_{sh}$ is not complemented in $\B_{sh}$. Otherwise we would have a real bounded projection $E:\B_{sh} \longrightarrow \k_{sh}$, then we can define a bounded projection 
$\tilde{E}: \B \longrightarrow \k$, $\tilde{E}(x)= -iE (i\Re e (x)) +  i   E(i\Im m (x))$, a contradiction. 

Let $q_1,q_2$ two infinite rank orthogonal projections on $\h$. We claim that $q_1\k_{sh}q_2$ is not complemented in $q_1\B_{sh} q_2$. In fact, suppose that there exists a real bounded projection $E: q_1\B_{sh} q_2 \longrightarrow q_1\k_{sh}q_2$. Let $v$ a partial isometry on $\h$ such that $v^*v=q_1$ and $vv^*=q_2$. Then we have that  $L_v E L_{v^*}: \mathcal{B}(q_2(\h))_{sh} \longrightarrow q_2 \k_{sh} q_2$ is a bounded projection, which is impossible by the previous paragraph.
\end{remarknf}

\medskip

\noi In the following result we collect the 	above proved properties of $\UK(P)$ and we give a complete characterization of the submanifold structure.

\begin{theorem}\label{compact quasi}
Let $P$ be the pinching operator associated with a family $\{  p_i \}_1 ^{w}$ $(1\leq w \leq \infty)$. Then $\UK(P)$ is a quasi submanifold of $\BK$. Furthermore, $\UK(P)$ is a  submanifold of $\BK$ if and only if $w < \infty$ and  there is only one infinite rank projection in the family $\{ \,  p_i \,\}_0 ^w$. 
\end{theorem}
\begin{proof}
The  first statement about the quasi submanifold structure of $\UK(P)$ has already been proved in Proposition \ref{tengent}  and Proposition \ref{topology ctoss}. Assume that $w < \infty$ and  there is only one infinite rank projection in the family $\{ \,  p_i \,\}_0 ^w$. 
The same proof of Theorem \ref{complemento} can be carried out
to show that $(T\UK(P))_P$ is complemented in $\BK$.

Suppose now that $\UK(P)$ is a submanifold of $\BK$. According to Proposition \ref{sous varietes}, there is a bounded linear projection 
$E: \BK \longrightarrow (T\UK(P))_P.$
Two cases should be considered: first, that   there are two infinite rank projections in the family $\{ \,   p_i \, \}_0 ^w$, and second, that $w=\infty$. In the first case, let $q_1 \in \{  \, p_0 , p_1 , \ldots , p_w   \, \}$ be an infinite rank projection and $q_2 \in \{  \, p_1 , \ldots , p_w \, \} \setminus  \{\, q_1 \, \}$ be other infinite rank projection.   In the second case, we set $q_1=\sum_{k=0}^{\infty}p_{2k}$ and $q_2=\sum_{k=0}^{\infty}p_{2k +1}$.  In any case we define  the following bounded linear map
\[   \tilde{E}: q_1 \B_{sh} q_2 \longrightarrow q_1 \k_{sh} q_2 , \, \, \, \, \, \,  \,\tilde{E}(q_1xq_2)= (L_{q_1}  E )(\, [L_{q_1xq_2 + q_2 xq_1},P]\,)(q_2). \]

\noi We claim that $\tilde{E}$ is a projection onto $q_1 \k_{sh} q_2$. In fact, notice that for each $x \in  \B_{sh}$ there is $z \in \k_{sh}$ such that $E(\,[L_{q_1xq_2 + q_2xq_1},P] \,)=[L_z,P]$. In the case in which there are two infinite rank projections, note that
$$ \tilde{E}(q_1xq_2)=\displaystyle{q_1\sum_{i=1}^{w} (zp_i - p_iz)q_2p_i}=q_1(zq_2 -q_2z)q_2=q_1zq_2.$$
On the other hand, when $w=\infty$, 
$$ 
\tilde{E}(q_1xq_2)=\displaystyle{q_1\sum_{i=1}^{\infty} (zp_i - p_iz)q_2p_i} =q_1\displaystyle{\sum_{k=0}^{\infty}(zp_{2k +1}-p_{2k +1}z)p_{2k + 1}}=q_1z\displaystyle{\sum_{k=0}^{\infty}p_{2k+1}}=q_1zq_2.   
$$
This proves that the range of $\tilde{E}$ is contained in $p_1 \k_{sh}p_2$. Moreover, let $x \in \k_{sh}$, then we have that $E(\,[L_{q_1xq_2 + q_2xq_1},P] \,)=[L_{q_1xq_2 + q_2xq_1},P]$. We thus get that  $$\tilde{E}(q_1xq_2)=q_1(q_1xq_2 + q_2xq_1)q_2=q_1xq_2.$$ 
Hence $\tilde{E}$ is a continuous linear projection onto $q_1\k_{sh}q_2$. In other words, $q_1\k_{sh}q_2$ is complemented in $q_1\B_{sh}q_2$, but this contradicts Remark \ref{rem complemnt}.
\end{proof}

\section{Covering map}\label{OI}

For $u \in \UI$, consider the inner automorphism given by $Ad_u: \I \longrightarrow \I$, $Ad_u(x)=uxu^*$.   Given   a pinching operator $P$ associated with a family $\{ \, p_i \, \}_1^w$, there is another orbit of $P$ defined by
$$\o_{\I}(P):=\{ \, Ad_u  P  Ad_{u^*} \, : \, u \in \UI \, \}.$$
Note that all the operators in $\o_{\I}(P)$ are pinching operators while $P$ is the only  pinching operator in $\UIP$. 
The isotropy group of the  the co-adjoint action is given by
\begin{equation}\label{characterization}
  H=\{ \, u \in \UI \, : \, Ad_u P Ad_{u^*}= P \, \}. 
\end{equation}
\noi In order to find a characterization of the operators in $H$ we need  the following lemma. We make the convention $\{  \, 0, \, 1, \ldots , \infty \}=\NN_0$.

\begin{lemma}\label{pinching equality}
Let $P$ be the pinching operator associated with a family $\{  \, p_i  \, \}_1 ^w$ and $Q$ be the pinching operator associated with another family $\{  \, q_i  \, \}_1 ^v$. Then $P=Q$ if and only if $w=v$ and $p_i=q_{\sigma(i)}$ for some permutation $\sigma$ of $\{  \, 0, \ldots , w \}$ such that $\sigma(0)=0$.
\end{lemma} 
\begin{proof}
We first suppose that $P=Q$. This is equivalent to 
\begin{equation}\label{suma} \sum_{i=1}^w p_ixp_i=\sum_{j=1}^v q_j x q_j , \end{equation}
for all $x \in \I$. If $\rank(p_i)< \infty$, $i\geq 1$, we set $x=p_i$ to get $\sum_{j=1}^v q_j p_i q_j=p_i$. Then it follows that $q_j p_i=q_jp_iq_j=p_i q_j$ for all $j\geq 1$. If $\rank(p_i)= \infty$, we use the same idea with a sequence of projections $(e_n)_n$ such that $e_n \leq p_i$, $e_n \nearrow p_i$, to find that $q_je_n=e_nq_j$, which  implies that $q_jp_i=p_iq_j$. Since $p_0=1-\sum_{i=1}^w p_i$ and $q_0=1-\sum_{i=1}^v q_j$, we can conclude that $q_jp_i=p_iq_j$ for all $i,j \geq 0$.

Now we claim that for each $i\geq 0$, we can find a unique $\sigma(i)$ such that $p_i = q_{\sigma(i)}$. To this end, let $\xi \in R(p_i)$, $\xi \neq 0$, and note that $p_i\xi=\xi=\sum_{j=0}^vq_j \xi$. This implies that there is some $j:=\sigma(i)$ such that $q_j \xi \neq 0$. Then we see that $q_j \xi=q_j p_i \xi =p_i q_j \xi$. Now let $\eta \in R(p_i)$  and insert $x=\eta \otimes q_j \xi$ in equation (\ref{suma}). In case $i>0$ we find that $\eta \otimes q_j \xi= (q_j \eta)\otimes q_j \xi$. If $j=0$, then $\eta \otimes q_j \xi=0$. In particular, if we take $\eta=q_j \xi \neq 0$, we obtain a contradiction. Hence we must have $j>0$, so the equation $\eta \otimes q_j \xi= (q_j \eta)\otimes q_j \xi$ implies that $q_j \eta =\eta$. Since $\eta$ is arbitrary, we have $R(p_i)\subseteq R( q_j)$. In a similar way,  we may choose $\eta \in R(p_j)$ to obtain that $R(q_j)\subseteq R( p_i)$. Thus  $p_i=q_j$.

In case $i=0$, we need to show that $p_0=q_0$. Suppose that there exists some $j>0$ such that $q_j \xi \neq 0$. By the preceding paragraph we know that $q_j \xi \in R(p_0)$. Then we insert $x=(q_j \xi) \otimes q_j \xi$ in equation (\ref{suma}) to find that $0=(q_j \xi) \otimes q_j \xi$, and hence  $q_j \xi=0$, a contradiction. Thus we obtain that $\xi=\sum_{j=0}^vq_j \xi=q_0 \xi$, and consequently, $R(p_0)\subseteq R(q_0)$. Interchanging $p_0$ and $q_0$, we can conclude that $p_0=q_0$.
 Since $\{ \, q_j  \,\}_0^v$ is a mutually orthogonal family, $\sigma(i)$ is unique and our claim is proved.

In other words, we have proved the existence of  a map $\sigma:\{  \, 0, \ldots , w \} \to \{  \, 0, \ldots , v \}$ satisfying $p_i=q_{\sigma(i)}$ and $\sigma(0)=0$. Repeating the previous argument with $q_j$ in place of $p_i$, we can construct another map $\psi:\{  \, 0, \ldots , v \} \to \{  \, 0, \ldots , w \}$ such that  $q_j=p_{\psi(j)}$ and $\psi(0)=0$. But $p_i=q_{\sigma(i)}=p_{(\psi \sigma)(i)}$ and $q_j=p_{\psi(j)}=q_{(\sigma\psi)(j)}$, so we have that $\sigma\psi=\psi\sigma=1$. Hence,  $\sigma$ is a permutation and $w=v$.

In order to prove the converse, let $\sigma$ a permutation of $\{ \, 0, \ldots , w \, \}$, $P$ be the pinching operator associated with a family $\{  \, p_i  \, \}_1 ^w$ and $Q$ be the pinching operator associated with $\{  \, p_{\sigma(i)}  \, \}_1 ^w$. Since the case $w < \infty$ is trivial, we suppose $w=\infty$.
Set $e_k=\sum_{i=0}^kp_i$. For each $x \in \I$, since $x$ is compact, we find that $\|(1-e_k)x\|\to 0$. Note that for $k \geq 1$, $$\sum_{i=1}^{\infty}p_{\sigma(i)}e_kxp_{\sigma(i)}=\sum_{i=1}^kp_ixp_i=\sum_{i=1}^{\infty}p_ie_kxp_i .$$ 
Then we get
\[  \bigg\| \sum_{i=1}^{\infty} p_{\sigma(i)}xp_{\sigma(i)} - \sum_{i=1}^{\infty}p_ixp_i \bigg\| = 
\bigg\| \sum_{i=1}^{\infty} p_{\sigma(i)}(1-e_k)xp_{\sigma(i)} - \sum_{i=1}^{\infty}p_i(1-e_k)xp_i \bigg\| \leq 2 \|(1-e_k)x\| \to 0, \] 
which proves that $P=Q$. 
\end{proof}

\noi Let $P$ be  the pinching operator associated with a family $\{ \, p_i \, \}_1 ^w$ $(1\leq w \leq \infty)$. Let $F$ be the set of all the permutations $\sigma$ of $\{ \, 0, \ldots, w \,  \}$ such that $\sigma(i)= i$ for all but finitely many $i \geq 0$. Note that the definition of the set $F$ becomes unnecessary if $w < \infty$.
 We will need to consider permutations of a finite number of  finite dimensional blocks  with the same dimension such that fix zero, i.e. 
\[  \mathcal{F}:=\{  \, \sigma  \in F \, : \,  \sigma(0)=0, \, \rank(p_i)=\rank(p_{\sigma(i)})< \infty  \text{ if $\sigma(i) \neq i$}   \,  \}. \]
Let $( \xi_{i,j(i)})$    be an  orthonormal basis of $\h$ such that $( \xi_{i,j(i)})_{j(i)=1, \ldots, \rank(p_i)}$ is a basis of $R(p_i)$,  where $i=0, \ldots ,w$. 
For each $\sigma \in \mathcal{F}$, we define the following permutation block operator matrix:
\[    r_{\sigma} (\xi_{i,j(i)}):=\xi_{\sigma(i), j(\sigma(i))}, \, \, \, i=0, \ldots ,w, \, j(i)=1, \ldots , \rank(p_i). \]
Note that $\rank(r_{\sigma} -1)<\infty$, since $\sigma \in \mathcal{F}$. Hence, it follows that $r_{\sigma} \in \UI$ for any symmetrically-normed ideal $\I$.

\begin{example}
A simple example takes place when $\h=\CC^n$, $\rank(p_i)=1$ and $\sum_{i=1}^np_i=1$. Here the set of all the matrices of the form $r_{\sigma}$, $\sigma \in \mathcal{F}$, reduces to all the $n \times n$ permutation matrices.   According to our next result, $H$ has exactly $n!$ connected components in this example.
\end{example}

\noi Recall that from the proof of Proposition \ref{est homog} we know that the isotropy group $G$ at $P$ corresponding to the action given by the left  representation can be characterized as block diagonal unitary operators, i.e.
\[ G=\bigg\{ \, u \in \UI \, : \, \sum_{i=0}^wp_iup_i=u \,   \bigg\},  \] 
where $P$ is the pinching operator associated with a family $\{ \, p_i \, \}_1 ^w$.

\begin{lemma}\label{connected comp}
Let $H$ be the isotropy group defined in  (\ref {characterization}). Then,
\[ H=\bigcup_{\sigma \in \mathcal{F}} r_{\sigma} G,  \]
where  each set in the union is a connected component of $H$.
\end{lemma}
\begin{proof}
Let $u \in \UI$ such that $Ad_u P Ad_{u^*} =P$. According to Lemma \ref{pinching equality} it follows that 
$up_iu^*=p_{\sigma(i)}$ for some $\sigma$ permutation of $\{ \, 0, \ldots ,w   \, \}$ such that $\sigma(0)=0$. In particular, note that $p_jup_i=\delta_{j,\sigma(i)}\,p_{\sigma(i)}u$, which actually says that $u$ has only one nonzero block in each row. Since  $u -1 \in \I$, we get that $\sigma \in \mathcal{F}$. Hence we can write $u=r_{\sigma} r_{\sigma^{-1}}u$, where $r_{\sigma^{-1}}u \in G$.

To prove the other inclusion it suffices to note that $r_{\sigma}up_iu^*r_{\sigma^{-1}}=r_{\sigma}p_ir_{\sigma^{-1}}=p_{\sigma(i)}$ for any $u \in G$. Then we apply again Lemma \ref{pinching equality} to obtain that $Ad_u P Ad_{u^*}=P$. 

In order to establish the last assertion about the connected components of $H$, we  remark that
\[    \| r_{\sigma}u - r_{\sigma'}v \|_{\I} \geq  \| r_{\sigma}u - r_{\sigma'}v \| \geq 1,  \] 
whenever $\sigma \neq \sigma'$ and $u,v \in G$. This implies that the distance between any pair of sets that appear in the union is greater than one. On the other hand, it is a well known fact that $\UI$ is connected, then so does $r_{\sigma}G$. Hence the lemma is proved.   
\end{proof}

\begin{remarknf}
As a consequence of Lemma \ref{connected comp}, $H$ is a Banach-Lie subgroup of $\UI$. Indeed, the connected components  of $H$ are diffeomorphic to the Banach-Lie subgroup $G$ of $\UI$. 

Hence it follows that $\o_{\I}(P)\simeq \UI /H$ has a manifold structure endowed with the quotient topology.
\end{remarknf}

\medskip

\begin{theorem}\label{covering map1}
Let $\Phi$ a symmetric norming function, and $\I=\mathfrak{S}_{\Phi}$.  Let $P$ be the pinching operator associated with a family $\{  p_i \}_1 ^{w}$. 
 If $\I \neq \k$ assume in addition that $w < \infty$ and  there is only one infinite rank projection in the family $\{ \,  p_i \,\}_0 ^w$.
Then the map
\[ \Pi: \UI(P) \longrightarrow \o_{\I}(P), \, \, \, \, \, \, \, \, \Pi(L_uPL_{u^*})=Ad_u P Ad_{u^*}\, , \]
is a covering map, when $\UI(P)$ is considered with the topology inherited from $\BI$ and  $\o_{\I}(P)$ with the quotient topology.
\end{theorem}
\begin{proof}
In the case where $\I \neq \k$, under the above hypothesis on the family $\{ \,   p_i \, \}_1 ^w$, it was proved in Lemma \ref{topology} that the quotient topology coincides with the subspace topology  on $\UI(P)$. In case $\I=\k$ both topologies coincide without  additional hypothesis  by Proposition \ref{topology ctoss}.
On the other hand, by Lemma \ref{connected comp} the quotient $H/G$ is discrete, then  $H/G$ is homomorphic to $\mathcal{F}$. 
We  define an action of $\mathcal{F}$ on $\UI(P)$ given by $\sigma \cdot L_uPL_{u^*}=L_{ur_{\sigma}}P L_{r_{\sigma^{-1}}u^*}$.
Therefore we can make the following identifications:
\[   \UI(P) / \mathcal{F} \, \simeq  \,  \UI(P) / (H/G)  \,  \simeq (\UI /G)/(H/G)  \,  \simeq \UI / H  \, \simeq  \, \o_{\I}(P).                      \]
Thus we may think of $\Pi$ as the quotient  map $ \UI(P) \longrightarrow  \UI(P) / \mathcal{F}$. Hence to prove that $\Pi$ is a covering map, it suffices to show that $\mathcal{F}$ acts properly discontinuous on $\UI(P)$ (see \cite{Gr}). This means that for any $Q \in \UI(P)$, there is an open neighborhood $\mathcal{W}$ of $Q$ such that $\mathcal{W} \cap \sigma \cdot \mathcal{W} = \emptyset$ for all $\sigma \neq 1$. Clearly, there is no loss of generality if we prove this fact for $Q=P$. 
To this end, define the open neighborhood by
\[  \mathcal{W}:=\{ \, Q \in \UI(P)  \, : \, \| Q - P \|_{\BI} < 1/2 \,  \}.  \]
Suppose that $\mathcal{W} \cap \sigma \cdot \mathcal{W} \neq \emptyset$ for some $\sigma \neq 1$. Then there are $Q,\tilde{Q} \in \mathcal{W}$ such that $\tilde{Q}=\sigma\cdot Q$. If $Q=L_uPL_{u^*}$, then we have that $\tilde{Q}=L_{ur_{\sigma}}PL_{r_{\sigma^{-1}}u^*}$. The distance between $Q$ and $\tilde{Q}$ can be estimated as follows
\[  \| Q -\tilde{Q}    \|_{\BI} = \| P - L_{r_{\sigma}} P L_{r_{\sigma^{-1}}}  \|_{\BI} \geq \bigg\| \sum_{i=1}^w  (p_i-p_{\sigma(i)}) (\xi \otimes \xi) p_i  \bigg\|_{\I} = \| \xi \otimes \xi \|_{\I} =1, \] 
where $\xi \in R(p_i)$ is such that $\| \xi \|=1$ and $\sigma(i)\neq i$. But since $Q, \tilde{Q} \in \mathcal{W}$, it follows that  $\|Q- \tilde{Q}\|_{\BI} <1$, a contradiction. Hence the action is properly discontinuous, and the proof is complete.
\end{proof}




\section{A complete Finsler metric} 
\noi Let $\Gamma(t)$, $t \in [0,1]$, be a piecewise $C^1$ curve in $\UI$. One can measure the length of $\Gamma$ using the norm of the symmetrically-normed ideal, i.e.
\[L_{\UI} (\Gamma) =
	\d\int_{0}^{1} \|\dot{\Gamma}(t)\|_{\I} \, dt.
\]
Since the tangent space of $\UI$ at $u$ can be identified with $u \I_{sh}$ (or also with $\I_{sh}u$), the above length functional is well defined. There is  rectifiable distance on $\UIH$ defined in the standard fashion, namely
	\[d_{\UI} (u_{0}, u_{1}) =
	\inf \left\{
	L_{\I} \left(\Gamma\right): \ 
	\Gamma \subseteq \UIH, \ 
	\Gamma(0) = u_{0}, \ 
	\Gamma(1) = u_{1}\right\}.
\]
Let $P$ be the pinching operator associated with a family  $\{ \, p_i \,  \}_1 ^w$. Since $\UI(P)$ is a homogeneous space, it becomes natural to put a quotient metric on the tangent spaces. If $Q=L_uPL_{u^*}$ for some $u \in \UI$, then for  $[L_z , Q] \in (T\UI(P))_Q$ we set
\begin{equation}\nonumber
\| \, [L_z ,Q] \, \|_Q = \inf \{\, \| z + y \|_{\I} \, : \, y \in \I_{sh}, \, Ad_u P Ad_{u^*}(y)=y      \,   \}.
\end{equation}
Indeed, the norm on $(T\UI(P))_Q$ is the Banach quotient norm of $\I_{sh}$ by the Lie algebra of the isotropy group at $Q$. A
standard computation shows that this metric is invariant under the action.  We point out that this quotient Finsler metric was already used in  several homogeneous spaces. For instance, we refer the reader to    \cite{AL9,ALR}, where some features of this metric are developed.  

The quotient Finsler metric on $\UIP$ allow us to introduce another length functional, namely 
\[  L_{\UI(P)}(\gamma)=\int_0 ^1 \| \dot{\gamma}(t) \|_{\gamma},   \] 
where    $\gamma(t)$, $t \in [0,1]$,  is a continuous and  piecewise $C^1$ curve in $\UIP$.  Thus there is an associated rectifiable distance given by
\[  d_{\UIP}(Q_0 ,Q_1)=   \inf \{  \,    L_{\UI(P)}(\gamma)          \, : \, \gamma \subseteq \UI(P), \, \gamma(0)=Q_0 , \, \gamma(1)=Q_1               \,    \},            \] 
when the curves $\gamma$ considered are continuous and piecewise $C^1$.
The next result proves that the rectifiable distance in
$\UIP$ can be approximated by lifting curves to $\UIH$. It is borrowed and adapted from \cite{AL9}.  


\begin{lemma}\label{lemma5.4.1}
Let $Q_{0}, Q_1 \in \UIP$. Under the assumptions of Theorem \ref{covering map1},
\[d_{\UIP} (Q_{0}, Q_{1} ) =
\inf \left\{
L_{\UI} (\Gamma):\ 
\Gamma \subseteq \UIH, \ 
L_{\Gamma(0)}Q_{0} L_{\Gamma(0)^*} = Q_{0}, \, 
L_{\Gamma(1)}Q_{0} L_{\Gamma(1)^*}  = Q_{1}
\right\},
\]
where the curves $\Gamma$ considered are
continuous and piecewise $C^{1}$.
\end{lemma}
\begin{proof}
Clearly it suffices to assume that $Q_0=P$. Let $\gamma(t) \in \UIP$ be a $C^{1}$ curve
joining $\gamma(0) = P$ and $\gamma(1) = Q_{1}$.
By Proposition \ref{est homog} the map
	\[ \pi : \UIH \rightarrow \UIP,
	\hspace{0.7cm}
\pi (u) = L_u P L_{u^*}
\]
is a submersion when $\UI(P)$ is endowed with the quotient topology, then
there exists a continuous piecewise smooth
curve $\Gamma$ in $\UIH$ such that
$\pi (\Gamma(t)) = \gamma(t)$, for all
$t \in [0,1]$.
From the  definition of the quotient Finsler metric, it is clear that the differential map of  $\pi$ at the identity given by
	\[\delta : \I_{sh} \rightarrow \left(T \UIP\right)_{P},
	\hspace{0.7cm}
\delta (z) = L_z P - P L_{z}
\]
is contractive. Moreover, since the action is isometric, the differential map of $\pi$ at any $u \in \UI$ has to be contractive. Using these facts we find that
	\[d_{\UIP}(P, Q_{1}) \leq
	L_{\UIP}(\pi(\Gamma)) \leq
	L_{\UI} (\Gamma).
\]
To complete the proof,
we must show that one can approximate
$L_{\UIP}(\gamma)$ with lengths of curves in
$\UIH$ joining the fibers of
$P$ and $Q_{1}$.
Fix $\epsilon >0$.
Let
$0 = t_{0} < t_{1} < \ldots <  t_{n} =1$
be a uniform partition of $\left[0,1\right]$
$\left(\Delta t_{i} = t_{i} - t_{i-1} =1 / n \right)$
such that the following hold:
\begin{enumerate}
	\item $\left\| \dot{\gamma} (s) - \dot{\gamma} (s') \right\|_{\BI} <
	\epsilon /4$
	if $s$, $s'$ lie in the same interval $\left[t_{i-1}, t_{i}\right]$.
	\item $\left|L \left(\gamma\right) - 
	\d\sum_{i=0}^{n-1} \left\|\dot{\gamma} \left(t_{i}\right) \right\|_{\gamma \left(t_{i}\right)}
	\Delta t_{i} \right| < \epsilon / 2 $.
\end{enumerate}
For each
$i=0, \ldots, n-1$, let
$x_{i} \in \I_{sh}$ be such that
	\[\delta_{\gamma\left(t_{i}\right)}\left(x_{i}\right) = \dot{\gamma}\left(t_{i}\right)
	\  \  \  \y \  \  \ 
	\left\|x_{i}\right\|_{\I} \leq
	\left\|\dot{\gamma} \left(t_{i}\right)\right\|_{\gamma \left(t_{i}\right)} +
	\epsilon / 2.
\]
Consider the following curve $\Gamma$ in $\UIH$:
	\[\Gamma(t) =
	\left\{
	\begin{array}
	[c]{ll}%
	e^{tx_{0}}                                       &     t \in \left[0,t_{1}\right), \\
	e^{\left(t-t_{1}\right)x_{1}} e^{t_{1}x_{0}}     &     t \in \left[t_{1},t_{2}\right),\\
	e^{\left(t-t_{2}\right)x_{2}} e^{\left(t_{2}-t_{1}\right)x_{1}} e^{t_{1}x_{0}}     &     t \in \left[t_{2},t_{3}\right),\\
	\ldots                                           & \ldots \\
	e^{\left(t-t_{n-1}\right)x_{n-1}} \ldots e^{\left(t_{2}-t_{1}\right)x_{1}} e^{t_{1}x_{0}}     &     t \in \left[t_{n-1},1\right].
	\end{array}
	\right.
\]
Clearly $\Gamma$ is continuous and piecewise smooth,
$\Gamma(0) =1$ and
	\[L_{\UI}(\Gamma)=
	\d\sum_{i=0}^{n-1} \left\|x_{i}\right\|_{\I} \Delta t_{i} \leq
	\left(\d\sum_{i=0}^{n-1}\left\|\dot{\gamma} \left(t_{i}\right)\right\|_{\gamma \left(t_{i}\right)} +
	\epsilon/2\right) \Delta t_{i} \leq
	L_{\UIP}(\gamma) + \epsilon.
\]
Let us show that
$\pi (\Gamma(1))$ lies close
to $Q_{1}$.
Indeed, first denote by
$\alpha (t) = \pi \left(e^{tx_{0}}\right) - \gamma(t)$,
then $\alpha(0) = 0$
and, using the mean value theorem in Banach spaces,
	\[\left\|\pi \left(e^{t_{1}x_{0}}\right)-\gamma\left(t_{1}\right)\right\|_{\BI} =
	\left\|\alpha \left(t_{1}\right) - \alpha(0)\right\|_{\BI} \leq
	\left\|\dot{\alpha}\left(s_{1}\right)\right\|_{\BI} \Delta t_{1},
\]
for some $s_{1} \in \left[0, t_{1}\right]$.
Explicity,
	\[\left\|\pi \left(e^{t_{1}x_{0}}\right) - \gamma\left(t_{1}\right)\right\|_{\BI} \leq
	\left\|L_{e^{s_{1}x_{0}}} \delta_{Q_{0}}\left(x_{0}\right) L_{e^{-s_{1}x_{0}}}-
	\dot{\gamma}\left(s_{1}\right)\right\|_{\BI} \Delta t_{1}.
\]
Note that
$\delta \left(x_{0}\right) = \dot{\gamma} (0)$,
and that
	\[\left\| L_{e^{s_{1}x_{0}}} \dot{\gamma}\left(0\right) L_{e^{-s_{1}x_{0}}}-
	\dot{\gamma}\left(s_{1}\right)\right\|_{\BI} \leq
	\left\|L_{e^{s_{1}x_{0}}} \dot{\gamma}\left(0\right) L_{e^{-s_{1}x_{0}}}-
	\dot{\gamma}\left(0\right)\right\|_{\BI}+
	\left\|\dot{\gamma}\left(0\right)-\dot{\gamma}\left(s_{1}\right)\right\|_{\BI}.
\]
The second summand is bounded by
$\epsilon /4$.
The first summand can be bounded as follows
	\[\left\|L_{e^{s_{1}x_{0}}} \dot{\gamma}\left(0\right) L_{e^{-s_{1}x_{0}}}-
	\dot{\gamma}\left(0\right)\right\|_{\BI} = 
	\left\|
	L_{e^{s_{1}x_{0}}} \dot{\gamma}\left(0\right) \left(L_{e^{-s_{1}x_{0}}} - I\right) +
	\left(L_{e^{s_{1}x_{0}}}-I\right)\dot{\gamma}\left(0\right)\right\|_{\BI} 
\]
	\[\leq 2 \left\|\dot{\gamma}\left(0\right)\right\|_{\BI}
	\left\|L_{e^{s_{1}x_{0}}} - I\right\|_{\BI}
	\leq 2 M \Delta t_{1},
\]
where
$M = \d\max_{t \in [0,1]} \left\|\dot{\gamma}(t)\right\|_{\BI}
.$
It follows that
	\[ \left\|\pi \left(e^{t_{1} x_{0}}\right) - \gamma \left(t_{1}\right)\right\|_{\BI} \leq
	\left(2 M \Delta t_{1} + \epsilon/4 \right) \Delta t_{1}.
\]
Next estimate
$\left\|\pi \left(e^{\left(t_2 - t_{1}\right) x_{1}}e^{t_{1} x_{0}}\right) - 
\gamma \left(t_{2}\right)\right\|_{\BI}$,
which by the triangle inequality is less or equal than
	\[\left\| L_{e^{\left(t_2 - t_{1}\right) x_{1}}e^{t_{1} x_{0}}} P
	L_{e^{-t_{1} x_{0}} e^{-\left(t_2 - t_{1}\right) x_{1}}} -
	L_{e^{\left(t_2 - t_{1}\right) x_{1}}} \gamma\left(t_1\right) 
	L_{e^{-\left(t_2 - t_{1}\right) x_{1}}}\right\|_{\BI}
\]
	\[+ \left\|
	L_{e^{\left(t_2 - t_{1}\right) x_{1}}} \gamma\left(t_1\right) L_{e^{-\left(t_2 - t_{1}\right) x_{1}}} -
	\gamma \left(t_2\right)\right\|_{\BI}.
\]
The first summand is
	\[\left\| L_{e^{\left(t_2 - t_{1}\right) x_{1}}e^{t_{1} x_{0}}} P
	L_{e^{-t_{1} x_{0}} e^{-\left(t_2 - t_{1}\right) x_{1}}} -
	L_{e^{\left(t_2 - t_{1}\right) x_{1}}}
	\gamma\left(t_1\right)
	L_{e^{-\left(t_2 - t_{1}\right) x_{1}}}\right\|_{\BI} =
\]
	\[= \left\|
	L_{e^{\left(t_2 - t_{1}\right) x_{1}}}
	\left(L_{e^{t_{1} x_{0}}}P L_{e^{-t_{1} x_{0}}} - \gamma \left(t_1\right)\right)
	L_{e^{-\left(t_2 - t_{1}\right) x_{1}}}\right\|_{\BI}=
\]
	\[= \left\| L_{e^{t_{1} x_{0}}} P L_{e^{-t_{1} x_{0}}} - \gamma \left(t_1\right) \right\|_{\BI} \leq
	\left(2M \Delta t_{1} + \epsilon/4 \right) \Delta t_{1}.
\]
The second can be treated analogously as the
first difference above,
	\[ \left\|
	L_{e^{\left(t_2 - t_{1}\right) x_{1}}} \gamma\left(t_1\right)
	L_{e^{-\left(t_2 - t_{1}\right) x_{1}}} -
	\gamma \left(t_2\right)\right\|_{\BI} \leq
	\left(2M \Delta t_{2} + \epsilon / 4 \right) \Delta t_{2}.
\]
Thus (using that
$\Delta t_{i} = 1/n$)
	\[ \left\|\pi \left(e^{\left(t_{2} - t_{1}\right) x_{1}}
	e^{t_{1} x_{0}} \right) -
	\gamma \left(t_{2}\right) \right\|_{\BI} \leq
	\left(2M/n + \epsilon /4 \right) 2/n.
\]
Inductively,
one obtains that
	\[\left\| \pi \left(\Gamma \left( t_{n-1}\right)\right) -
	\gamma \left(t_{n-1}\right)
	\right\|_{\BI} \leq
	2 \left( 2M/n + \epsilon/4\right) <
	\epsilon / 2
\]
choosing $n$ appropriately. According to Lemma \ref{topology}, when $\I \neq \k$,  or according to Proposition \ref{topology ctoss}, when $\I=\k$,   the map $\pi$ has continuous local cross sections. Then one can connect $\Gamma(t_{n-1})$ with the fiber of $Q_1$ with a curve of arbitrary small length.
\end{proof}

\noi In order to prove our next theorem, we need to
state the next lemma
(see \cite[p. 109]{takesaki}):

\begin{lemma}\label{lemma3.4.1}
Let $H$ be a metrizable topological group,
and $G$ be a closed subgroup.
If $d$ is a complete distance function on $H$
inducing the topology of $H$, and
if $d$ is invariant under the right translation by $G$,
i.e.,
$d\left(xg, yg\right) = d\left(x,y\right)$
for any $x$, $y$ $\in H$ and
$g \in G$,
then the left coset space
$H/G = \left\{x G : \  x \in H \right\}$
is a complete metric space under the metric
$\dot{d}$ given by
	\[\dot{d} \left(xG, yG\right) =
	\inf \left\{d\left(xg_1, yg_2\right): \ g_1,g_2 \in G \right\}.
\]
Moreover,
the distance $\dot{d}$ is a metric for the quotient topology.
\end{lemma}

\noi We will make use of the former lemma with $H=\UI$ and $G$ the isotropy group at $P$. 

\begin{theorem}
Let $\Phi$ a symmetric norming function, and $\I=\mathfrak{S}_{\Phi}$. Let $P$ be the pinching operator associated with a family $\{  p_i \}_1 ^{w}$. 
 If $\I \neq \k$ assume in addition that $w < \infty$ and  there is only one infinite rank projection in the family $\{ \,  p_i \,\}_0 ^w$. 
Let $u,v \in \UI$, and let
	\[\dot{d}_{\UI}
	\left(L_{u} P L_{u^{\ast}}, L_{v} P L_{v^{\ast}}\right) =
	\inf\left\{d_{\UI}\left(uv_1, v v_{2}\right): \  v_{1}, v_2 \in G \right\}.
\]
Then,
$\dot{d}_{\UI} = d_{\UIP} $.
In particular,
$(\UIP, \  d_{\UIP})$
is a a complete metric space and $d_{\UIP}$
metricates the quotient topology.
\end{theorem}
\begin{proof}
We begin by recalling  that $(\UI, d_{\UI})$ is a complete metric space and $G$ is $d_{\UI}$-closed in $\UIH$ (see e.g. \cite[Lemma 2.4]{C}). Thus the quotient distance
$\dot{d}_{\UI}$
is well defined.
Moreover,
since the multiplication by unitaries is isometric,
it can be computed as
	\[\dot{d}_{\UI}
	(L_{u} P L_{u^{\ast}}, L_{v} P L_{v^{\ast}}) =
	\inf\left\{d_{\UI}(u, v v_{1}): \  v_{1} \in G \right\}.
\]
To prove one inequality, fix
$\epsilon > 0$.
By Lemma \ref{lemma5.4.1} there is
a curve $\Gamma \in \UIH$
satisfying
\begin{enumerate}
	\item $\Gamma (0) = u$, $\Gamma (1)=v v_{1}$, with $v_{1} \in G$,
	\item $L_{\UI} (\Gamma) <
	d_{\UIP}\left( L_{u} P L_{u^{\ast}}, L_{v} P L_{v^{\ast}}  \right)+ \epsilon$.
\end{enumerate}
Then we have that
	\[\dot{d}_{\UI}
	(L_{u} P L_{u^{\ast}}, L_{v} P L_{v^{\ast}}) \leq
	d_{\UI} \left(u, v v_{1}\right) \leq
	L_{\UI} \left(\Gamma\right) <
	d_{\UIP} \left(L_{u} P L_{u^{\ast}}, L_{v} P L_{v^{\ast}}\right) + \epsilon.
\]
Since $\epsilon$ is arbitrary, we have proved the first inequality.
To show the reversed inequality,
note that given $\epsilon > 0$, there
exists $v_{1} \in G$ such that
	\[d_{\UI} (u, v v_{1}) < 
	\dot{d}_{\UI}
	(L_{u} P L_{u^{\ast}}, L_{v} P L_{v^{\ast}}) + \epsilon
\]
Then there exists a curve $\Gamma \subseteq \UIH$
such that
$\Gamma (0) = u$,
$\Gamma (1) = v v_{1}$ and
$L_{\I} \left(\Gamma\right) <
	d_{\I}\left(u, v v_{1}\right) + \epsilon.
$
So we have that
	\[d_{\UIP} (L_{u} P L_{u^*}, L_{v} P L_{v^*}) \leq
	L_{\UI} \left(\Gamma\right) <
	d_{\UI} \left(u, vv_{1}\right) + \epsilon <
	\dot{d}_{\UI}\left(L_{u} P L_{u^{\ast}}, L_{v} P L_{v^{\ast}}\right) + 2 \epsilon.
\]
We thus get $\dot{d}_{\UI} = d_{\UIP}$.
The completeness of
$(\UIP, \  d_{\UIP})$
and the fact that $d_{\UIP}$ defines
the quotient topology,  
follow   from
Lemma \ref{lemma3.4.1}.
\end{proof}

\section{Application to the unitary orbit of a compact normal operator}
 
 Let $a$ be a compact normal operator.  The question of when the full unitary orbit of $a$, i.e.
 \[  \U(a)= \{  \, uau^*   \,  : \, u \in \U   \,   \}, \]
has the property that the quotient topology coincides with the uniform norm topology was completely solved by L. A. Fialkow \cite{Fi}. Both topologies coincide if and only if $a$ has finite rank. In this section, we  address the same question but with respect to the $\UI$-unitary orbit of $a$, which is given by 
\[  \U_{\I}(a)=\{ \, uau^*  \, : \, u \in \UI  \, \}.    \] 
Though the $\UI$-unitary orbit is in general smaller than the full unitary orbit, both orbits are equal if $a$ has finite rank (see \cite[Lemma 2.7]{La}). Recall that for $u \in \UI$,  
$$uau^*=a  + a(u^*-1) + (u-1)au^* \in a + \I.$$  Thus one can  endow $\U_{\I} (a)$ with the topology inherited from the affine Banach space $a + \I$.  On the other hand, as a homogeneous space, $\UI (a)$ may also be endowed with the quotient topology. 

If $\I$ is ideal of the trace class operators, it was proved by P. Bon\'a \cite{Bo} that both topologies coincide when $a$  has  finite rank. Later this result was extended to any symmetrically-normed ideal   by  D. Belti\c t$\breve{\text{a}}$ and T. Ratiu in \cite[Theorem 5.10]{BR}, where they also showed that the $\UI$-unitary orbits are weakly Kh$\ddot{\text{a}}$ler homogeneous spaces. We will show the converse of this result and  we will  give a different proof of the already known implication by means of the previous results on the orbits of pinching operators. 

Our result is also related to the work by E. Andruchow, G. Larotonda and L. Recht \cite{AL9, ALR, La}, where  without the assumption of $a$ being compact, several equivalent conditions to the existence of a submanifold structure of the  $\UI$-unitary (or full unitary) orbits are described, when $\I$ is  the ideal of Hilbert-Schmidt or compact operators. In particular, they established sufficient conditions to ensure that both topologies coincide.  One of this conditions states that the spectrum of $a$ must be finite. Note that this gives again the sufficient condition, since if $a$ is compact, the spectrum of $a$ is finite if and only if $a$ has finite rank.

\medskip

\begin{remarknf}\label{relationship a and P}
The main idea to link unitary orbits of pinching operators  with the $\UI$-unitary orbit of a compact operator is the following. By the spectral theorem we may rewrite the compact normal operator $a$ as a uniform norm convergent series, namely
\begin{equation}\label{spectral}
  a=\sum_{i=1}^{w} \lambda_i p_i, 
\end{equation}  
where $1 \leq w \leq \infty$, $\lambda_i$ are the nonzero distinct eigenvalues of $a$ and $\{   \, p_i \, \}_1 ^w$ is a family of mutually orthogonal finite rank projections. Indeed, $p_i$ is the orthogonal projection onto $\ker(a-\lambda_i)$. Then we take $P$ to be  the pinching operator associated with $\{   \, p_i \, \}_1 ^w$. 

Let $u \in \UI$ such that $ua=au$. If we use the spectral decomposition of $a$, we see that $u$ must be block diagonal with respect to the family $\{   \, p_i \, \}_0 ^w$. This says that the isotropy group at $a$ coincides with the isotropy group at $P$, i.e.  
\[  \{  \, u \in \UI  \, : \, ua=au \, \}=\{  \, u \in \UI  \, : \, L_uP=PL_u \, \}=G.  \]
Hence it turns out that the quotient topology on $\U_{\I} (a)\simeq \UI /G$ is equal to the quotient topology on  $\UI(P)$.
 \end{remarknf}

\medskip

\begin{theorem}
Let $\Phi$ a symmetric norming function, and $\I=\mathfrak{S}_{\Phi}$. Let $a$ be a compact normal operator. Then the quotient topology on $\UI(a)$ coincides with the topology inherited from $a+\I$  if and only if $\rank(a) < \infty$.
\end{theorem}
\begin{proof}
Suppose that $\rank(a) < \infty$. This is equivalent to state that $w< \infty$ in the spectral decomposition of $a$ given by equation (\ref{spectral}). Under this assumption the family $\{ \, p_i \, \}_0 ^w$ has only one projection of infinite rank, namely $p_0= 1 - \sum_{i=1}^w p_i$. Indeed, note that $p_0$ is  the orthogonal
 projection onto $\ker(a)$. According to Proposition \ref{topology} when $\I \neq \k$, or Proposition \ref{topology ctoss} when $\I=\k$, the quotient topology coincides with the topology inherited from $\BI$ on $\UI(P)$.  

Since the quotient topology on $\UI(a)$ is always stronger than the topology inherited from $a + \I$, it remains to prove that any sequence
$(u_n)_n$  in $\UI$ satisfying $\| u_n a u_n^* -a \|_{\I} \to 0$ has to be convergent to $a$ in the quotient topology. To this end, note that
\[    \| p_iu_n p_j  \|_{\I}  \leq |\lambda_i - \lambda _j|^{-1} \| u_n a  -a u_n \|_{\I} \to 0,  \]
for all $i,j \geq 0$ and $i\neq j$ (where we set $\lambda_0=0$) . Now let $x \in \I$ such that $\|x\|_{\I}=1$.  Since 
\[  \bigg\|  \sum_{i=1} ^w  (u_n p_i -p_i u_n) x p_i \bigg\|_{\I} \leq \sum_{i=1} ^w \|u_n p_i -p_i u_n\|_{\I} \leq 2 \sum_{i\neq j}  \|p_ju_n p_i\|_{\I},    \]  
we see that
\[   \|L_{u_n}PL_{u_n ^*} - P\|_{\BI} = \|L_{u_n}P - PL_{u_n}\|_{\BI}\leq 2 \sum_{i\neq j}  \|p_ju_n p_i\|_{\I} \to 0. \]
By the remarks in the first paragraph of this proof and Remark \ref{relationship a and P}, the latter is equivalent to say that $u_nau_n^* \to a$ in the quotient topology.

\smallskip

In order to prove the converse we assume that the quotient topology on $\UI(a)$ coincides with the topology inherited from $a+\I$. We need to consider two cases. In the first case we suppose that $\I \neq \k$. Let $\mathcal{M}$ be the supplement of the Lie algebra of $G$ defined in Proposition \ref{est homog}. If $\rank(a)=\infty$, we can construct a sequence $(z_k)_k$ in  $\mathcal{M}$  such that $\|z_k\|\to 0$ and $\|z_k\|_{\I} =1$ (see Remark \ref{sequence prop}). 

Given $\epsilon >0$, let $M \geq 1$ such that $\|\sum_{i=M+1}^w \lambda_i p_i \| \leq \epsilon$. Then it follows that
\begin{align*}
\| e^{z_k} a e^{-z_k} - a \|_{\I} & = \| (e^{z_k} -1) a  - a(e^{z_k} -1) \|_{\I} \\ 
& \leq 2 \bigg( \, \| e^{z_k} -1 \| \bigg\| \sum_{i=1}^M \lambda_i p_i \bigg\|_{\I} + 
\|e^{z_k} -1\|_{\I} \bigg\|\sum_{i=M+1}^w p_i \bigg\| \, \, \bigg) \\
& \leq 2 \bigg( \, \| e^{z_k} -1 \| \bigg\| \sum_{i=1}^M \lambda_i p_i \bigg\|_{\I} + 
e \,\epsilon \, \bigg).  
\end{align*}
Letting $k\to \infty$, we find that $e^{z_k} a e^{-z_k} \to a$ in the norm $\|  \, \cdot \, \|_{\I}$, or equivalently, in the quotient topology. 
By the same argument used at the beginning of Lemma \ref{topology} we can arrive at $\|z_k\|_{\I} \to 0$, a contradiction with our previous choice of $(z_k)_k$. 

Now we turn to the case where $\I=\k$. Under the assumption that both topologies coincide on $\UK(a)$ we claim that the map       
\[   \Lambda:\UK(a) \longrightarrow \UK(P), \, \, \, \, \, \,     \Lambda(uau^*)=L_uPL_{u^*},  \]
is continuous, when one endows $\UK(a)$ with the topology inherited from $ \k$ and $\UK(P)$ with the topology inherited from $\BK$. In fact, by Proposition \ref{topology ctoss} the quotient and the inherited topologies always coincide on $\UK(P)$. Then the map  $\Lambda$ turns out to be  the identity map of $\UK /G$, and thus our claim follows.

Again we suppose that $\rank(a)=\infty$. We will find a contradiction with the fact that $\Lambda$ is continuous. Note that there must be an infinite number of finite rank projections in the family $\{ \, p_i \, \}_1 ^w$ and the eigenvalues of $a$ satisfy $\lambda_i \to 0$.    Let $( \xi_{i,j(i)})$    be an  orthonormal basis of $\h$ such that $( \xi_{i,j(i)})_{j(i)=1, \ldots, \rank(p_i)}$ is a basis of $R(p_i)$ for all $i\geq 1$.
Then take the following sequence of unitary operators:
\[   u_n= \xi_{n+2 ,1} \otimes \xi_{n+1,1} + \xi_{n+1 ,1} \otimes \xi_{n+2,1} + e_n , \] 
where $e_n$ is the orthogonal projection onto $\{ \, \xi_{n+1,1} \, , \, \xi_{n+2 ,1} \, \}^{\perp}$. Note that $u_n - 1$ has finite rank, then $u_n \in \UK$.
Thus we get
\begin{align*}
 \|u_nau_n^* - a \| & =\| u_n a -au_n \| \\ 
 & =  \| (\lambda_{n+1} - \lambda_{n+2}) \, (\xi_{n+2 ,1} \otimes \xi_{n+1,1}) - (\lambda_{n+2} - \lambda_{n+1}) ( \xi_{n+1 ,1} \otimes \xi_{n+2,1} \, )  \| \\
 & \leq 2 | \lambda_{n+1} - \lambda_{n+2} | \to 0.
\end{align*}
On the other hand, note that
\begin{align*}
\| L_{u_n}PL_{u_n^*} -P \|_{\BK} & = \sup_{\|x\|=1 \, , \, x \in \k} \bigg\|  \sum_{i=1}^{\infty}(u_np_i -p_i u_n) x p_i \bigg\| \\
& \geq \bigg\|  \sum_{i=1}^{\infty}(u_np_i -p_i u_n)  p_i \bigg\| \\
& = \| \xi_{n+2 ,1} \otimes \xi_{n+1,1}\, + \, \xi_{n+1 ,1} \otimes \xi_{n+2,1}  \|= 1,
\end{align*} 
which contradicts the continuity of $\Lambda$. Hence $a$ must have finite rank, and the theorem is proved.
\end{proof}

\vspace{1.2cm}

\noi
\begin{tabular}[h]{ll}
			 Eduardo Chiumiento &  Mar\'ia Eugenia Di Iorio y Lucero \\
			 Departamento de Matem\'atica, FCE-UNLP & Instituto de Ciencias, UNGS\\
			 Calles 50 y 115 & J. M. Gutierrez 1150\\
			 (1900) La Plata, Argentina & (1613) Los Polvorines\\
			  \textit{e-mail}: \texttt{eduardo@mate.unlp.edu.ar} &  Buenos Aires, Argentina\\
			   & \textit{e-mail}: \texttt{mdiiorio@ungs.edu.ar}\\
			  & \\
			  & \\
			  E. Chiumiento and M. Di Iorio y Lucero &\\
        Instituto Argentino de Matem\'atica& \\
        ``Alberto P. Calder\'on'', CONICET & \\
        Saavedra 15  Piso 3& \\
                (1083) Buenos Aires, Argentina & \\
\end{tabular}

\end{document}